\documentclass[a4paper]{elsarticle} 

\usepackage{url}
\usepackage[latin1]{inputenc}
\usepackage{graphicx}
\usepackage{epsfig}
\usepackage{fancyhdr,fancybox}
\usepackage{amsfonts}
\usepackage{amsmath,amsthm}
\usepackage{amssymb}
\usepackage{verbatim}
\usepackage{subfigure}
\usepackage{latexsym}
\usepackage{multirow}
\usepackage{color}

\newcommand{\eps}{\varepsilon}
\newcommand{\Rd}{\mathbb R^d}

\newcommand{\Fe}{F_\eps}
\newcommand{\R}{\mathbb R}
\newcommand{\Z}{\mathbb Z}
\newcommand{\N}{\mathbb N}
\newcommand{\E}{\mathbb E}

\newcommand{\var}{\textrm{var}}
\renewcommand{\S}{\mathbf{S}}
\renewcommand{\qed}{$\hfill\square$}
\newcommand{\vol}{\textrm{vol}}
\newcommand{\K}{\mathcal K^d}
\newcommand{\ring}{\mathcal R^d}

\newcommand{\tow}{\stackrel{P}{\longrightarrow}}
\newcommand{\tod}{\stackrel{d}{\longrightarrow}}
\newcommand{\sC}{{\mathcal{C}}}
\newcommand{\asC}{\overline{\mathcal{C}}}
\newcommand{\sM}{{\mathcal{M}}}
\newcommand{\Unp}{\textrm{Unp}}
\newcommand{\reach}{\textrm{reach}}
\newcommand{\Ind}{{\bf 1} }
\newcommand{\tr}{\mathrm{tr}}
\newcommand{\esslim}[1]{\underset{#1}{\mathrm{esslim}\,}}

\newtheorem{theorem}{Theorem}[section]
\newtheorem{lemma}[theorem]{Lemma}
\newtheorem{lemma*}{Lemma}

\newtheorem{corollary}[theorem]{Corollary}

\newtheorem{definition*}[lemma*]{Definition}
\newtheorem{example}[theorem]{Example}

\newtheorem{remark}[theorem]{Remark}

  \let\oldmarginpar\marginpar
  \renewcommand\marginpar[1]
    {\-\oldmarginpar[\raggedleft\footnotesize #1]{\raggedright\footnotesize #1}}

\begin{document}
\title{Estimation of fractal dimension and fractal curvatures from digital
images}
\author{Evgeny Spodarev}
\ead{evgeny.spodarev@uni-ulm.de}
\address{Ulm University, Institute of Stochastics, 89069 Ulm, Germany}
\author{Peter Straka}
\ead{p.straka@unsw.edu.au}
\address{University of New South Wales, School of Mathematics and
Statistics, Sydney NSW 2052, Australia}
\author{Steffen Winter\corref{mycorrespondingauthor}}
\cortext[mycorrespondingauthor]{Corresponding author} 
\ead{steffen.winter@kit.edu}
\address{Karlsruhe Institute of Technology, Department of Mathematics, Kaiserstr.\ 89, 76128 Karlsruhe, Germany}



\begin{abstract}
Most of the known methods for estimating the fractal dimension of fractal sets are based on the evaluation of a single geometric characteristic, e.g.~the volume of its parallel sets.  We propose a method involving the evaluation of several geometric characteristics, namely all the intrinsic volumes (i.e.\ volume, surface area, Euler characteristic etc.) of the parallel sets of a fractal.
Motivated by recent results on their limiting behaviour, 
we use these functionals to estimate the fractal dimension of sets from digital images.
Simultaneously, we also obtain estimates of the \emph{fractal curvatures} of these sets, some fractal counterpart of intrinsic volumes, allowing a finer classification of fractal sets than by means of fractal dimension only.
We show the consistency of our estimators and test them on some digital images of self-similar sets.
\end{abstract}

\maketitle

\section{Introduction}

For the classification of fractal sets, it is
common to examine their fractal dimension.
Two major algorithms for estimating fractal dimensions
are well known and extensively used: The box counting algorithm and the
\emph{sausage method}, whose name is due to the use of
$\eps$-parallel sets
\begin{align}\label{eq:parallelset}
F_\eps := \{x\in \Rd : d(x,F)\leq \eps\}, \quad \eps\geq 0,
\end{align}
for the approximation of a given fractal set $F\subset \Rd$ (see
e.g.\ \cite{stoyan1994frs}). The latter approach is related
to the \emph{Minkowski-dimension} $\dim_M F$, which is the number $s\geq 0$ such that
\begin{align}\label{eq:dimM-informal}
\vol_d(F_\eps)\sim c\cdot\eps^{d-s}, \text{ as } \eps\searrow 0
\end{align}
for some constant $c$. (Here and throughout $\sim$ means that the
quotient of left and right hand side converges to $1$,
$\vol_d(\cdot)$ is the $d$-dimensional Lebesgue measure and
$d(x,F):=\inf_{y\in F} |x-y|$, where $|\cdot|$ is the Euclidean norm
in $\Rd$). The Minkowski dimension is known to be equivalent to
the box counting dimension for any bounded set $F\subset\R^d$.
Thus both approaches estimate the same mathematical
object.
Beside these two most popular algorithms, many other computation methods such as
e.g.\ the local dimension method, and various refinements of the two basic methods
are available, see e.g.
\cite{cutler1989eds,martinezlopez2001ime,plotnick1996lag,
sandau1997mfd,taylor91,MR2963995} as well as the books \cite{stoyan1994frs, MR1390993} and the references
therein.

It has been observed in many applications that often the fractal dimension alone is not
sufficient to distinguish or classify different fractal structures and additional \emph{texture} parameters have been suggested. One of the simplest and most prominent is the \emph{lacunarity}, suggested by Mandelbrot in \cite{Mandelbrot82,mandelbrot94}. For a set $F\subset\R^d$ with Minkowski dimension $\dim_M F=s$ it is defined as the reciprocal of its \emph{$s$-dimensional Minkowski content} $\sM^s(F):=\lim_{\eps\searrow 0} \eps^{s-d}\vol_d(F_\eps)$. That is, the lacunarity is essentially one over the constant $c$ in the relation \eqref{eq:dimM-informal}. In the sausage method, the $y$-intercept of the regression line is a reasonable estimator for $\log c$ and therefore, one gets the lacunarity almost for free together with the estimate for the dimension $s$. Lacunarity can be understood as a measure of how fast the space around the fractal is occupied when the parallel set grows. It is able to distinguish fractal structures of equal dimension. Algorithms for determining lacunarity and related  texture parameters have been proposed and analyzed, e.g.~in \cite{MR1130096,MR2435388,Backes1,Backes2} and used successfully in very different fields such as pattern recognition \cite{Backes1}, signal processing \cite{MR3014287}, DNA classification \cite{MR3105523}, the analysis of aggregation clusters in statistical mechanics \cite{DLAcluster} and of breast tumors in medicine \cite{MR3106985}, to mention just a few recent studies.
Despite these many positive examples, one can not hope that a single texture parameter like lacunarity will always be able to successfully distinguish or classify fractal structures of a given class. It is easy to construct sets with very different texture and visual appearance but the same dimension and lacunarity, from which the need for more texture parameters is evident.

In the present work, we suggest a whole vector of texture parameters for fractal sets based on the recently introduced concept of fractal curvatures \cite{diss-winter} and we propose some methods how to estimate these parameters separately or simultaneously from given binary images of a fractal set. Our approach may be viewed as a generalization of the sausage method:
Given a bounded subset $F\subset\R^d$, we look at the parallel sets $F_\eps$ of $F$ for small radii $\eps>0$. While the sausage method considers the behaviour of the volume of these parallel sets only, we propose to study the asymptotic behaviour of all (or at least several) \emph{total curvatures} $C_0(F_\eps),\ldots, C_d(F_\eps)$ as $\eps\searrow 0$. 
Total curvatures (or \emph{intrinsic volumes}) are important geometric characteristics and are defined for different classes of bounded sets $A\subset\R^d$, e.g.\ convex and polyconvex sets, sets of positive reach and their unions, etc. They have the following interpretations: $C_d(A)$ is the $d$-dimensional volume, $C_{d-1}(A)$
is essentially the surface area and $C_0(A)$ is the Euler characteristic of $A$. The remaining functionals have interpretations as integrals of mean curvature. For
a set $A\in\R^2$ this means for instance that $C_2(A)$, $C_1(A)$ and $C_0(A)$
are (up to normalization) area, boundary length
and Euler characteristic of $A$, respectively. The intrinsic volumes can be determined simultaneously from binary images using e.g.~the algorithms described in \cite{guderleietal}.

The proposed methods are based on the following ideas and definitions from \cite{diss-winter}:
For a fractal set $F\subset\R^d$  with (Minkowski) dimension $\dim_M F=s$, the $k$-th total curvature $C_k(F_\eps)$ behaves typically
like $\eps^{k-s}$  as $\eps\searrow 0$. Often, for instance for non-arithmetic self-similar sets, one has direct proportionality, that is, the limit
\begin{align}
\label{eq:frac-curv}
{\sC}_k(F)&:=\lim_{\eps\searrow 0}{\eps^{s-k}C_k(F_\eps)}
\end{align}
exists and is then called the \emph{$k$-th fractal curvature} of $F$.
In case the limit in \eqref{eq:frac-curv} does not exist,
 {e.g.\ for the Sierpinski gasket,} $C_k(F_\eps)$ may still be of the order $\eps^{s-k}$. Then one often observes oscillations in the geometry and hence
 in the total curvatures $C_k(F_\eps)$ which do not vanish as $\eps\searrow 0$.
 Instead $C_k(F_\eps)$ is asymptotic to some periodic function. This is the typical behaviour for instance for arithmetic self-similar sets.
 In this case one has $C_k(F_\eps)= \Theta( \eps^{k-s})$ as $\eps\searrow 0$, that is, the quotient $|C_k(F_\eps)|/\eps^{k-s}$ is bounded from above and below by some constants.
 Moreover, the following average limit typically exists:
\begin{align} \label{eq:avg-frac-curv}
\asC_k(F)&:=\lim_{\delta\searrow 0}\frac{1}{|\log\delta|}
  \int_\delta^{1}\eps^{s-k}C_k(F_\eps)\frac{d\eps}{\eps},
\end{align}
which is then called the \textit{$k$-th average
fractal curvature} of the set $F$.
For $k=d$, the definitions \eqref{eq:frac-curv} and \eqref{eq:avg-frac-curv} are just the well known Minkowski content and its averaged counterpart, which are thus naturally included in the framework of (average) fractal curvatures.  We refer to Section~\ref{sec:fractal} and the references therein for more details and results on fractal curvatures.

Based on these ideas, we propose two methods for estimating
simultaneously the fractal dimension and the (average) fractal
curvatures of a given set $F\subset\R^d$ from its digital
approximations. The first method is based on a multivariate linear
regression and tries to estimate simultaneously all (or at least several) fractal curvatures in \eqref{eq:frac-curv} together with the dimension. This attempt does only make sense under the assumption that for $F$ all the fractal curvatures exist which are to be included in the regression.
Since in many situations this assumption is not satisfied, even for
self-similar sets, we propose a second method, which tries to estimate averaged
fractal curvatures instead.
The second method is a more sophisticated quasi--linear regression
inspired by a time series approach with a linear drift and a truncated
Fourier series as a seasonal part to model the oscillations in the
geometry. It allows to estimate the  average fractal curvatures even when
the limits in \eqref{eq:frac-curv} do not exist. For this approach to be
meaningful, the assumption on the existence of fractal curvatures is replaced
by the much weaker assumption that the expression $\eps^{s-k} C_k(F_\eps)$ is asymptotic to some
(multiplicatively) periodic function $p_k(\eps)$ as $\eps\searrow 0$. This is what is typically observed in situations where some kind of self-similarity is present.

Roughly speaking, the estimation procedure in the first method is
as follows: Given a binary image of a fractal set $F\subset\R^d$,
we first measure the values of $C_k(F_{\eps_j})$ for a set of
dilation radii $\{\eps_1,\ldots,\eps_n\}$ and all $k \in
\{0,\ldots,d\}$.  In this step, we employ the algorithm described
in \cite{minkdiscr} and \cite{guderleietal} which allows for a
simultaneous computation of \emph{all} intrinsic volumes in only
one scan of each set $F_{\eps_j}$. Second, we use the $(d+1)$
asymptotic relations
\begin{align}\label{eq:1}
C_k(F_\eps) &\sim \sC_k(F) \eps^{k-s}, \quad \text{ as } \eps \searrow 0,
\end{align}
implied by \eqref{eq:frac-curv} for a linear regression. Multiplying by $\eps^{-k}$ and taking logarithms of the absolute values on both sides in \eqref{eq:1}, we get the relation $\log \left(\eps^{-k} |C_k(F_{\eps})|\right)\sim \beta_k - s \log \eps$ as $\eps\searrow 0$, where $\beta_k:=\log |\sC_k(F)|$, which suggests to compare
the expression $\log \left(\eps^{-k} |C_k(F_{\eps})|\right)$ to the line $\beta_k+sx$ in the variable $x:=-\log\eps$. Similarly, by combining all the data, the set of vectors
$$
\left\{\left(\log (\eps_j^{-0} |C_0(F_{\eps_j}))|,\log (\eps_j^{-1} |C_1(F_{\eps_j}))|,\ldots,\log
(\eps_j^{-d} |C_d(F_{\eps_j}))|\right)\right\}_{j=1,\ldots, n}
$$
plotted against
$x_j=-\log \eps_j$  {provides} a point cloud in $\R^{d+2}$
that resembles a line and a least squares fit will result in an estimate of
the dimension $s=\dim_M F$ of the fractal set $F$, as well as of
its fractal curvatures $\sC_k(F)$, $k=0,\ldots,d$. Deviations from the line are due to image
discretization, measurement errors and the described geometric oscillations
of the intrinsic volumes that may only vanish as $\eps
\searrow 0$.  These errors are supposed to be random. They
can not be observed directly. Notice that this is the only source
of randomness in this method, since the set $F$ is deterministic.
Due to these assumptions, statistical regression methodology can be
used.

In the second method, the linear regression step is replaced by a quasi--linear regression, which can be interpreted as fitting a periodic function to the above point cloud.
For details of both methods we refer to Section~\ref{sec:estimator}.

Under suitable assumptions on the covariance structure of the error in our models and for suitable choices of the radii $\eps_j$ we prove the
weak consistency of our estimators as the number of observations
$n$ tends to $\infty$. Furthermore, we have implemented the
algorithms and tested them on a number of self-similar sets.

 The
paper is organized as follows: In Section~\ref{sec:fractal}, some
notions from fractal geometry are recalled and the relevant results on
curvature measures and fractal curvatures are reviewed. In Section~\ref{sec:estimator},
we introduce the two methods for estimating the fractal dimension $s(F)$
and the fractal curvatures $\sC_k(F)$, $k=0,\ldots,d$ of a fractal
set $F$ and discuss their asymptotic properties, based on a suitable model for the discretization errors.
Section~\ref{sec:numerical} is concerned with the implementation of the methods and some simulation results:
For some self-similar sets in $\R^2$ the fractal dimension and the fractal curvatures are estimated
using the proposed methods and the results are compared to the (known) exact
dimension and fractal curvatures as well as to estimates of the dimension provided by conventional methods.


\section{Fractal dimension and fractal curvatures}
\label{sec:fractal}

In this section we provide some theoretical background required
for the justification of our approach. First we recall a few facts
on fractal dimensions and self-similar sets. Then we discuss
curvature measures and review some recent results on fractal
curvatures.
\ \\
\paragraph{\bf Box counting dimension and Minkowski dimension.}
For a bounded set $F\subset\R^d$ and $\eps>0$, recall the definition of the $\eps$-parallel set $F_\eps$  from \eqref{eq:parallelset}.
The  number
\begin{align*}
\dim_M F:=d-\lim_{\eps\searrow 0}\frac{\log \vol_d(F_\eps)}{\log \eps}
\end{align*}
is called the \emph{Minkowski dimension} of $F$, provided the limit
exists (cf.~\cite{falconer90}). It is well known that the
Minkowski dimension of any set $F$ coincides with its \emph{box
dimension} $\dim_B F$ (provided one of these numbers exists), which is defined by
\begin{align*}
  \dim_BF:= \lim_{\delta\searrow 0}\frac{\log N_\delta(F)}{-\log \delta}.
\end{align*}
Here $N_\delta(F)$ is the number of boxes in a $\delta$-grid of
$\R^d$ that intersect $F$. Moreover, $\dim_M F$ is always an upper
bound for the Hausdorff dimension $\dim_H F$ of $F$. See
e.g.~\cite{falconer90} for more details on fractal dimensions and
their properties and interrelations.
\ \\
\paragraph{\bf Self-similar sets.}
Let $S_i:\R^d\rightarrow \R^d$, $i=1,\ldots,N$, be contracting
similarities. Denote the contraction ratio of $S_i$
 by $r_i\in(0,1)$. 
It is a well known fact (cf.\ \cite{hutchinson81}),
 that for a system $\{S_1,\ldots,S_N\}$ of similarities there exists a unique,
 non-empty, compact subset $F$ of $\R^d$ satisfying the invariance relation $\S(F)=F$, where $\S$ is
the set mapping defined by
\begin{align*}\S(A)= \bigcup_{i=1}^{N}S_i(A), \quad A\subseteq\R^d.\end{align*}
$F$ is called the \emph{self-similar set} generated by the system
$\{S_1,\ldots,S_N\}$. Moreover, the unique solution $s$ of
$\sum_{i=1}^N r_i^s=1$ is called the \emph{similarity
dimension} of $F$.

The system $\{S_1,\ldots,S_N\}$ is said to satisfy the \emph{open set
condition} (OSC) if there exists an open, non-empty, bounded
subset $O\subset\R^d$
 such that
 \begin{align*}
 S_i O \subseteq O \text{ for } i=1,\ldots,N \text{ and } S_i O \cap
S_j O=\emptyset \text{ for all } i\neq j.
\end{align*}
The OSC ensures that the images $S_1F,\ldots, S_N F$ of $F$ do not overlap too much. It is well known, that for self-similar sets $F$ satisfying OSC, all the different dimensions coincide, i.e.~one has $\dim_M F = \dim_H F = s$, where $s$ is the similarity dimension of $F$.  For sets not satisfying OSC much less is known; still $s$ is an upper bound for  Hausdorff and Minkowski dimension, but these two may be strictly smaller than $s$ and, furthermore, they might differ. In the sequel we shall assume that the self-similar sets satisfy OSC.

Let $h>0$. A finite set of positive numbers $\{y_1,...,y_N\}$ is
called \emph{$h$-arithmetic} if $h$ is the largest number such
that $y_i\in h\Z$ for $i=1,\ldots, N$. If no such number $h$
exists for $\{y_1,...,y_N\}$, the set is called
\emph{non-arithmetic}. We attribute these properties to the system
$\{S_1,\ldots,S_N\}$ or to $F$ if the set $\{-\log
r_1,\ldots,-\log r_N\}$ has them. In this sense, each self-similar
set $F$ is either $h$--arithmetic for some $h>0$ or
non-arithmetic. Sierpi\'nski carpet and  Sierpi\'nski gasket are
$\log 2$-- and $\log 3$--arithmetic. For further examples of
arithmetic and non-arithmetic sets see
Section~\ref{sec:numerical}.
\ \\
\paragraph{\bf Curvature measures and intrinsic volumes.}
We first recall the notion of curvature measures for polyconvex
sets, as this class of sets includes the parallel sets of
digitized sets and hence is a sufficiently general setting for the
implementation of our algorithms. For a general definition of
fractal curvatures we briefly discuss curvature measures on more
general classes of sets in the next paragraph.

Let $\K$ denote the class of compact, convex subsets of $\Rd$ and $\ring$ the class of sets that can be represented as finite unions of sets in $\K$.
$\ring$ is called the \emph{convex ring}, its elements are \emph{polyconvex sets}.
For sets $K\in\K$, the volume $\vol_d$ of the $\eps$-parallel sets of $K$ is given by the so called \emph{Steiner formula}. 
For $\eps\ge 0$, $\vol_d(K_\eps)$ is a polynomial in $\eps$:
\begin{equation}
 \label{eq:steiner}
 \vol_d(K_\eps) = \sum_{k=0}^d \eps^{d-k}\kappa_{d-k}C_k(K).
\end{equation}
The coefficients $C_0(K),\ldots, C_d(K)$ are called the \emph{intrinsic volumes} or \emph{total curvatures} of $K$. $\kappa_j$ denotes the $j$-dimensional volume of the unit ball in $\mathbb R^j$.

The total curvatures are the total masses of certain measures called the \emph{curvature measures} of $K$. They satisfy a \emph{local Steiner formula} which is due to Federer~\cite{federer1959cm}. Let $p_K$ denote the metric projection onto the set $K$. For any Borel set $A\subset\Rd$, the volume of the set $K_\eps \cap p_K^{-1}(A)$ is again a polynomial in $\eps$:
\begin{align}
\label{eq:local-steiner}
  \vol_d(K_\eps\cap p^{-1}_K(A)) = \sum_{k=0}^d \eps^{d-k}\kappa_{d-k} C_k(K,A).
\end{align}
The coefficients $C_0(K,\cdot),\ldots,C_d(K,\cdot)$ are measures in the second argument.
They are called the \emph{curvature measures} of $K$. Their total masses are the total curvatures, $C_k(K,\Rd) = C_k(K)$.
Curvature measures are \emph{additive}, i.e.~if $M, K$ and $M\cup K\in \K$, then $M\cap K\in \K$ and the curvature measures satisfy
\begin{align*}
C_k(M\cup K,\cdot) &=C_k(M,\cdot) + C_k(K,\cdot) - C_k(M\cap K,\cdot), \qquad k=0,\ldots,d.
\end{align*}
This allows to extend curvature measures additively to $\ring$. Iterating the above formula gives the
inclusion-exclusion formula for sets $K_1,\ldots, K_m\in \K$ such that their union $K:=K_1\cup\ldots\cup K_m$ is in $\K$:
\begin{align}
\label{eq:in-ex}
 C_k(K,\cdot)
=\sum_{I\subset \{1,\ldots,m\}}{(-1)^{|I|-1} C_k(\bigcap_{i \in I}
K_i,\cdot)},  \qquad k=0,\ldots,d.
\end{align}
Here $|I|$ denotes the cardinality of the set $I$. Now, if the
left hand side is not defined, i.e.~if $K$ is in $\ring$ but not
in $\K$, take the right hand side as its definition. Groemer
\cite{groemer:1978} has shown that this extension is well defined,
i.e.~that the left hand side is independent of the chosen
representation of $K$ by convex sets. In general, for $K\in\ring$
the curvature measure $C_k(K,\cdot)$ is a signed measure,
$k=0,\ldots, d-2$. Denote by $C_k^{\var}(K,\cdot)$ its total
variation and put $C_k^{\var}(K) := C_k^{\var}(K,\Rd)$,
$k=0,\ldots, d$.
\ \\
\paragraph{\bf Sets with positive reach.} For $X\subset\R^d$, let $\Unp(X)$ be the set of points $y\in\R^d$ which have a unique nearest point in $X$.
$\Unp(X)$ consists of those points for which the metric projection
onto $X$ is well defined. The supremum over all radii $\eps\geq0$
such that
 $X_\eps\subset \Unp(X)$ is called the \emph{reach} of $X$, $\reach(X)$, and $X$ is said to have \emph{positive reach}, if $\reach(X)>0$.
For any set $K$ with positive reach, the local Steiner formula \eqref{eq:local-steiner} holds for all $\eps$ such that $0<\eps<\reach(K)$,
which allows to define the curvature measures $C_0(K,\cdot),\ldots, C_d(K,\cdot)$ of $K$ just as before, see \cite{federer1959cm}.
These curvature measures have similar properties, in particular they are additive, motion invariant and homogeneous of degree $k$  (meaning $C_k(rK) = r^kC_k(K)$ for $r>0$).
They can be extended additively to finite unions of such sets, although some care is necessary as not all unions are feasible. Instead of
discussing this extension in detail, we address another extension, which is particularly useful in our situation namely for sets that are themselves parallel sets.

For a bounded set $K\subset\R^d$, a radius $\eps>0$ is called
\emph{regular}, if $\eps$ is a regular value of the distance
function of $K$ in the sense of Morse theory, see \cite{Fu}.
According to \cite{Fu}, for $d\leq 3$ and $K\subset\R^d$ a
compact set, almost all $\eps>0$ are regular for $K$. Regularity of a
radius $\eps$ for $K$ implies, that the boundary of $K_\eps$ is a
Lipschitz manifold and the closed complement $\widetilde{K_\eps}$
of $K_\eps$ has positive reach. Therefore, the curvature measures
of $\widetilde{K_\eps}$ are well defined in the sense of Federer
and the curvature measures of $K_\eps$ are then given by means of
the following reflection principle (see \cite{RZ}):
\begin{align*}
  C_k(K_\eps,\cdot)=(-1)^{d-k-1} C_k(\widetilde{K_\eps},\cdot),\qquad k=0,\ldots, d-1.
\end{align*}
As before, we denote by $C_k(K_\eps):=C_k(K_\eps,\R^d)$ the total
masses of the measures $C_k(K_\eps,\cdot)$, which are also called
the \emph{Lipschitz-Killing curvature measures} of $K_\eps$. We
continue to use the term \emph{total curvatures} for the
$C_k(K_\eps)$. Recall that $C_k^\var(K_\eps,\cdot)$ is the total
variation measure of the (in general signed) measure
$C_k(K_\eps,\cdot)$ and $C_k^\var(K_\eps)$ its total mass.
\ \\
\paragraph{\bf Scaling exponents and fractal curvatures.}
For the definition of fractal curvatures for a set, it is
necessary that sufficiently many of its close parallel sets admit
curvatures measures.
For a compact set $F\subset\Rd$, we assume that almost all $\eps>0$ are regular. 
(As mentioned above, in space dimensions $d\leq 3$, this is always satisfied.)
Then, for each $k\in\{0,\ldots,d\}$, the \emph{$k$-th curvature scaling exponent} $s_k=s_k(F)$ of the set $F$ is defined by
\begin{align}\label{eq:sk}
s_k(F) := \inf\left\{t\in\mathbb R:\esslim{\eps\searrow 0}\eps^tC_k^{\var}(F_\eps)=0\right\}, 
\end{align}
cf.\ e.g.\ \cite[p.13]{diss-winter} or \cite[eq.\ (1.5)]{PokW}. The typical value of $s_k(F)$ for a fractal set $F$ of dimension $\dim_M F=s$ is $s_k=s-k$. Although this relation may fail for certain sets $F$,
it is useful to concentrate on the following essential limit (avoiding the irregular $\eps$) and call it the \textit{k-th fractal curvature} of $F$ in case it exists:
\begin{eqnarray}
\label{eq:frac-curv2}
\sC_k(F)&:=&\esslim{\eps\searrow 0}{\eps^{s-k}C_k(F_\eps)}.
\end{eqnarray}
In general, the limit in \eqref{eq:frac-curv2} does not exist. Even for self-similar sets, it often
fails to exist. Therefore, the following Cesaro averaged
version of the limit is considered, which has a better convergence behaviour.
For $k=0, \ldots, d$, the \emph{$k$-th average
fractal curvature} of $F$ is the number
\begin{align}   \label{eq:avg-frac-curv2}
\asC_k(F) := \lim_{\delta\searrow0}\frac{1}{|\log\delta|}
\int_\delta^{1} \eps^{s-k}C_k(F_\eps) \frac{d\eps}{\eps}
 \end{align}
provided this limit exists. Note that if
$\sC_k(F)$ exists, then $\asC_k(F)$ exists as well and
both numbers coincide.
The functionals $\sC_k(F)$ and
$\asC_k(F)$ deserve to be called \emph{curvatures},
since they 
share some of the desirable properties of total
curvatures. In particular, they are motion-invariant and
homogeneous, though in general $\sC_k$ is of degree $k+s_k$, cf.~\cite{diss-winter}. As fractal curvatures are limits of classical total curvatures, they are expected to carry important geometric information about the fractal set $F$ and therefore they are natural candidates to be considered as geometric indices or texture parameters.


\begin{remark}
       The definition of the fractal curvatures here is slightly different to the one given in \cite{diss-winter}, where the exponent $s_k$ is put in general instead of $s-k$. This slightly changed point of view (taken up e.g.\ in \cite{Zaehle} and \cite{WZ}) emphasizes the generic case. It may give a zero for some fractal curvature in the exceptional cases where the exponent $s-k$ is not optimal. However, recent results in \cite{PokW} show that these exceptional cases are rare and can be classified completely at least in $\R$ and $\R^2$.
       In general, one could define for each $t\in\R$ the \emph{$t$-dimensional} $k$-th fractal curvature of a set $F$ by $\lim_{\eps\searrow 0}{\eps^{t}C_k(F_\eps)}$ (just in the same way as for the $t$-dimensional Minkowski content of $F$). Then the two different definitions would be special cases (which often coincide) of this general notion. A similar remark applies to average fractal curvatures.
\end{remark}

\paragraph{\bf The fractal curvatures of self-similar sets.}
In general it is difficult to determine the  fractal dimension or
the Minkowski content of a set exactly and it is even more
difficult for fractal curvatures. However, for self-similar sets
some rigorous results have been established, which exemplify that
the above definitions are reasonable and useful. So let now $F$ be
a self-similar set satisfying OSC.
In \cite{diss-winter}, self-similar sets with polyconvex parallel sets have been considered. Polyconvexity  is a rather restrictive condition,
which ensures however that curvature measures are well defined for \emph{all} parallel sets $F_\eps$ of $F$. Note also that polyconvexity is easy to check, as the following criterion holds, see
~\cite{llorentewinter07}: $F_\eps$ is
polyconvex for \emph{all} $\eps > 0$ if and only if there exists some $\eps_0
>0$ such that $F_{\eps_0}$ is polyconvex.

For such sets it was shown that for $k=0,\ldots,d$, the expression $\eps^{s-k}C_k^{\var}(\Fe)$ is uniformly bounded as $\eps\searrow 0$, implying in particular that $s-k$ is a general upper bound for the $k$-th scaling exponent $s_k$, that is, we have
\begin{align}
  s_k \leq s-k.
\end{align}
Moreover, a characterization was given of when (average) fractal curvatures exist for these sets:
\begin{theorem}
\cite[Theorem 2.3.6 and Remark 4.1.5]{diss-winter}
\label{fraccurv-mainthm}
Let $F\subset\R^d$ be a self--similar set generated by the system $\{S_1,\ldots,S_N\}$ with contraction ratios $r_i$ and similarity dimension $s$. Suppose that $F$ satisfies the OSC and has polyconvex parallel sets.
Then,
for each $k \in \{0,\ldots,d\}$, there exists a bounded function $p_k : (0,1]
\rightarrow \R$ such that
\begin{align*}
\eps^{s-k}C_k(F_\eps) \sim p_k(\eps) \quad \text{ as } \eps \searrow 0.
\end{align*}
Moreover, the following holds:
\begin{enumerate}
\item[(i)]
If $F$ is $h$-arithmetic (for some $h>0$), then $p_k$ can be chosen multiplicatively
periodic with period $h$, i.e.\ $p_k(h\eps) = p_k(\eps)$, 
and
$\overline \sC_k(F)$ exists.
\item[(ii)]
If $F$ is non-arithmetic, then $p_k$ can be chosen constant, and
$\sC_k(F)$ exists.
\end{enumerate}
The value of $\overline \sC_k(F)$ (and in the non-arithmetic case of $\sC_k(F)$) is given by the integral
\begin{equation}\label{eq:curv-formula}
\frac 1{\eta}
  \int_0^1 \eps^{s-k-1}
\left( C_k(F_\eps) - \sum_{i=1}^N \textbf{\textup{1}}_{(0,r_i]}
(\eps)C_k((S_iF)_\eps)\right)
{d}\eps,
\end{equation}
where  $\eta=-\sum_{i=1}^N r_i^s\log r_i$.
\end{theorem}

Note that the last formula allows explicit (but rather tedious)
calculations of $\overline \sC_k(F)$. These exact values will be
compared with the values estimated from binary images of
some fractals $F$ in Section~\ref{sec:numerical}.
It follows from \cite[Lemma 3.2]{llorentewinter07} that for an $h$-arithmetic self-similar set $F$ the value of $\asC_k(F)$
is equivalently given by
\begin{equation}\label{eq:Int_curv-formula}
\asC_k(F)= \frac 1{h_0}
  \int_0^{h_0} p_k\left(e^{-x}\right)\, {d} x,
\end{equation}
where $h_0=-\log h$.

Theorem~\ref{fraccurv-mainthm} extends to a more general class of self-similar sets. The polyconvexity can be replaced by the weaker regularity assumption mentioned above, that almost all $\eps$ are regular for $F$. In this general situation, one needs to assume additionally that a rather technical curvature bound is satisfied. We refer to \cite{Zaehle, RZ12} for further details. In \cite{Zaehle}, analogous results have  been established for random self-similar sets and in \cite{Kombrink,Bohl13} for self-conformal sets.
For the cases $k=d$ and $k=d-1$, none of these assumptions are necessary. As volume and surface area of the parallel sets are well defined for any set $F\subset\R^d$, the assertions of Theorem~\ref{fraccurv-mainthm} hold for any self-similar set satisfying OSC regardless of any polyconvexity or regularity assumption, see \cite{Gatzouras} for the case $k=d$ and \cite{ratwin1} for the case $k=d-1$. In these two cases it has also been shown that $\asC_k(F)$ (as well as $\sC_k(F)$, if it exists) are strictly positive. Moreover, some deep connections between $\sC_{d-1}(F)$ and the Minkowski content $\sC_d(F)$ have been established in \cite{ratwin1}. In particular, for any self-similar set $F\subset\R^d$ (with OSC) one has the equality $s_{d-1}=s_d-1=s-1$, provided $s<d$. Moreover, $\sC_{d-1}(F)$ and $\sC_{d}(F)$ coincide up to some normalisation constant:
\begin{align}\label{eq:d=d-1}
  \sC_{d-1}(F)=\frac{d-s}2 \sC_d(F).
\end{align}
The same relation holds for the average counterparts, see \cite[Theorems 4.5 and 4.7]{ratwin1}. In \cite{ratwin2} it is shown, that the relation \eqref{eq:d=d-1} holds in fact for arbitrary bounded sets $F\subset\R^d$ with $\dim_M F<d$, that is, whenever one of these two fractal curvatures exists (as a positive and finite number) then the other one exists as well and equation  \eqref{eq:d=d-1} holds.

\section{Estimators of dimension and fractal curvatures}
\label{sec:estimator}

\subsection{Least squares methods}
\label{subsec:estimators}

\paragraph{General assumptions.}
Let $F\subset\R^d$ be a fractal set satisfying the
following assumptions:
\begin{itemize}
\item[(A1)] The parallel sets $F_\eps$ of $F$ are sufficiently regular for curvature measures $C_0(F_\eps,\cdot), \ldots, C_d(F_\eps,\cdot)$ to be well defined for almost all $\eps>0$.
\item[(A2)] For each $k=0,\ldots,d$, the expression $C_k(F_\eps)$ (as a function of $\eps$) is either strictly positive or strictly negative.
\item[(A3)] \label{A3} The fractal curvatures $\sC_0(F),\ldots,\sC_d(F)$ exist, i.e., for each $k=0,\ldots,d$, the essential limit in \eqref{eq:frac-curv2} exists.
\end{itemize}

Assumption (A1) is a condition on the regularity of the parallel sets. The notion of fractal curvature does not make sense for $F$ if the curvature measures of its parallel sets are not defined in some way.
However, most sets that one can think of satisfy this assumption. As outlined in Section 2, in $\R^d$, $d\le 3$, almost all $\eps>0$ are regular for $F$ and therefore this condition is satisfied. In higher dimensions the condition may fail but the construction of counterexamples is difficult.  Therefore,
from the point of view of applications, assumption (A1) imposes no restriction. Note in particular that the parallel sets of the digitized
fractal images are always polyconvex such that curvature measures are well defined for all $\eps>0$.

Assumption (A2) is a technical condition needed simply to be able to work on the logarithmic scale. It is a serious restriction.
On the other hand this condition can easily be checked from the data. Single indices $k$ can simply be excluded from the estimation when $C_k(F_\eps)$
 fail to satisfy condition (A2), see Remark~\ref{rem:J} below. We point out that this assumption is always satisfied for the two uppermost indices $d-1$ and $d$: for any set $F\subset\R^d$ and any $\eps>0$, volume $C_d(F_\eps)$ and surface area $C_{d-1}(F_\eps)$ are always strictly positive. As the curvature measures $C_k(F_\eps,\cdot)$, $k\le d-2$ are signed measures in general, the total curvatures may assume negative values. Moreover, as $\eps\searrow 0$ the total curvatures may switch their sign infinitely many times.
Assumption (A2) ensures that we can estimate the absolute values of the fractal curvatures from the absolute values of the data and put the correct sign back to the estimated quantities in the end. 

 Assumption (A3) is the most restrictive but also the most  {natural} assumption. Without the existence of fractal curvatures, it would not make sense to try and estimate them. Note that implicitly we assume in (A3) that the k-th scaling exponent of $F$ is bounded from above by $s-k$, compare Section~\ref{sec:fractal},
where it is also outlined that a rigorous mathematical proof of the existence of fractal curvatures has up to now only been given
for non-arithmetic self-similar sets. Later on we shall replace (A3)  by the weaker assumption (A3'), which ensures that the average fractal curvatures exist, a setting which includes in particular all self-similar sets.

\paragraph{First method: Ordinary least squares.}
Assumption (A3) implies that the asymptotic relation \eqref{eq:1} holds for $k=0,\ldots,d$ which (after taking absolute values, multiplying by $\eps^{-k}$ and taking logs) can be written as
\begin{align} \label{eq:1-log}
\log(\eps^{-k}|C_k(F_\eps)|) &\sim \log|\sC_k(F)| - s \log \eps\,, \quad \text{ as } \eps \searrow 0,
\end{align}
for each $k=0,\ldots,d$. This allows for a linear regression.

Given a set of $\eps$-values $\eps_1,\ldots,\eps_n$, we set
\begin{align}
x_j &:= -\log \eps_j \label{eq:def-xj}\,, \\
  y_{kj}  &:= \log \left( \eps_j^{-k} |C_k(F_{\eps_j})|\right), \quad k=0,\ldots, d.
\label{eq:def-variables}
\end{align}
Relation \eqref{eq:1-log} suggests that if the radii $\eps_j$ are small enough, then the points $(y_{0j},y_{1j},\ldots, y_{dj})\in\R^{d+1}$, $j = 1,\ldots,n$ will lie close to a line 
(see Figure~\ref{fig:3d-plot}). Setting $\beta_k:=\log|\sC_k(F)|$, $k=0,\ldots,d$, we expect
\begin{equation*} 
\begin{pmatrix}
  y_{0 j} \\ y_{1 j} \\ \vdots \\ y_{d j}
\end{pmatrix}
=
\begin{pmatrix}
\beta_0\\ \beta_1\\ \vdots \\ \beta_d
\end{pmatrix}
+
s x_j
\begin{pmatrix}
   1 \\ 1 \\ \vdots \\  1
 \end{pmatrix}
+
\begin{pmatrix}
 \delta_{0j}   \\ \delta_{1j} \\ \vdots \\ \delta_{dj}
 \end{pmatrix}
 ,
\qquad  j=1,\ldots, n,
\end{equation*}
 whereas the
discretisation and computation errors $\{\delta_{kj}: \;
k=0,\ldots, d,\; j=1,\ldots,n \}$ are correlated random variables
with $\E \delta_{kj}=0$ and $\var \, \delta_{kj} \in (0,\infty)$
for all $k=0,\ldots, d$ and $j=1,\ldots,n$. Notice that we may
assume $\delta_{kj}$ to be random since a fractal $F$ and its
parallel sets can be digitized in many different ways. To do that,
it suffices to move $F$ with respect to the digitization lattice
of pixels (voxels) arbitrarily. The assumption $\E \delta_{kj}=0$
reflects the fact that fractal curvatures and dimension are motion
invariant. The errors are correlated because parallel sets are
monotone increasing with respect to inclusion: $F_{\eps_i} \subset
F_{\eps_j}  $ if $\eps_i < \eps_j$.

\begin{figure}[ht]
  \centering
\includegraphics[width=0.7\textwidth]{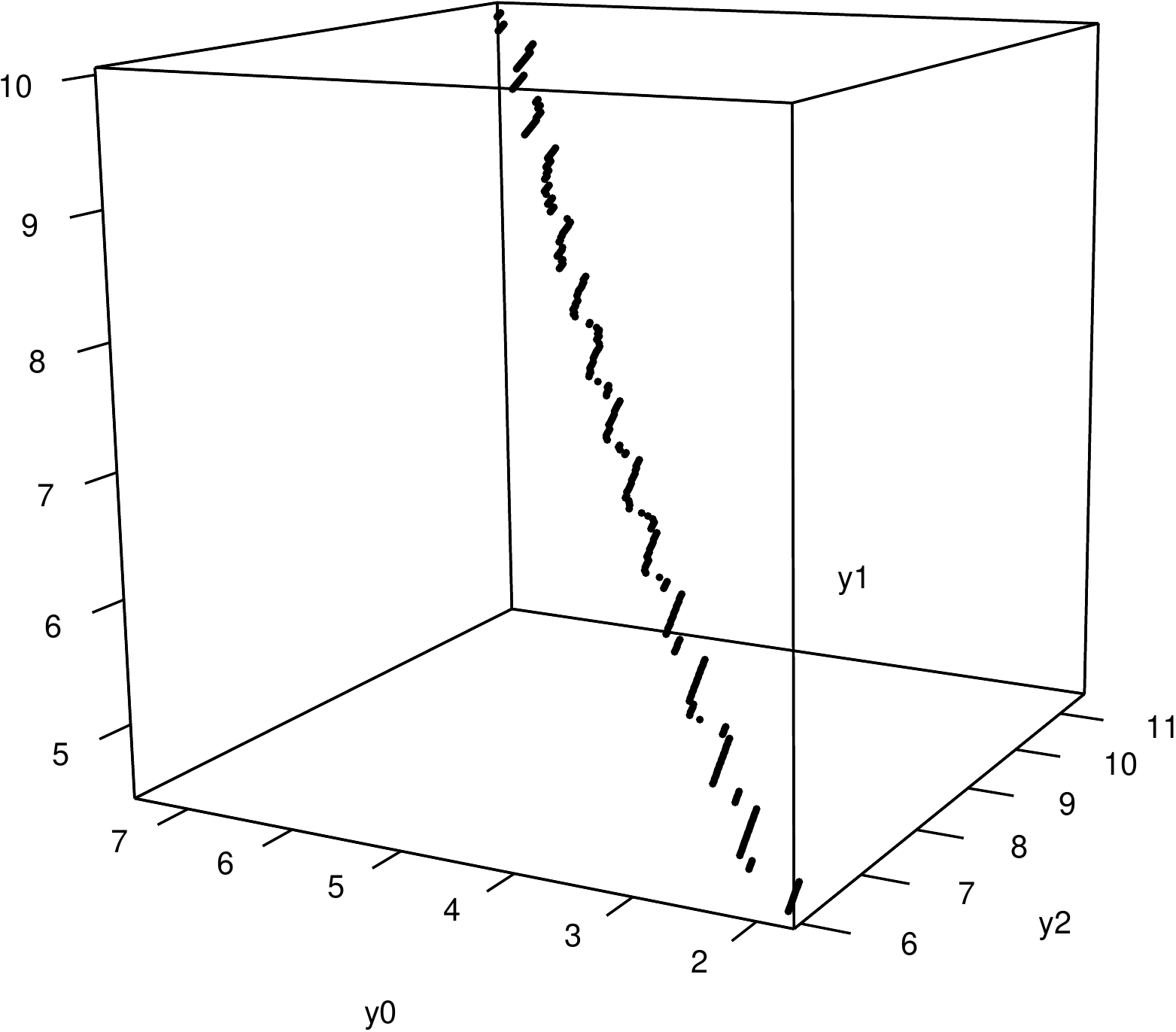}
\caption[Plot of point-cloud] {The point cloud $\{(y_{0j},y_{1j},y_{2j})\}_{j=1}^n$ for a (2-dimensional) $3000\times 3000$ binary image of the
 Triangle Set (see Figure~\ref{fig:thesets}(d)). The $352$ data points have been obtained for dilation radii ranging from
$2.73$ to $89.8$ [pixels], chosen in an equidistant way on a logarithmic scale.}
  \label{fig:3d-plot}
\end{figure}

To estimate the fractal dimension $s$ and the quantities $\beta_k$, $k=0,\ldots, d$ (which encode the fractal curvatures), we fit a line of the form
\begin{align*}
y=\begin{pmatrix}\beta_0 \\ \beta_1 \\ \vdots \\ \beta_d \end{pmatrix}
+s x \begin{pmatrix} 1 \\ 1 \\ \vdots\\ 1 \end{pmatrix}
\end{align*}
to the point cloud
$\{(y_{0j},y_{1j},\ldots, y_{d j})\}_{j=1}^n$ by the ordinary least squares method. That is, we find the values of $s$ and $(\beta_0,\ldots,\beta_d)$  for which the expression
\begin{equation}
\label{eq:S^2}
e_n^2\left(\beta_0,\ldots,\beta_d,s\right) := \frac{1}{(d+1)n} \sum_{k=0}^d \sum_{j=1}^n (y_{kj}-(\beta_k+  s x_j))^2
\end{equation}
is minimal. Standard least-squares calculations show that these
uniquely defined values are
\begin{eqnarray}
\label{eq:s}
\hat s^{(n)} &=& \dfrac{\sum_{k=0}^d  \sum_{j=1}^n y_{kj} 
    (x_j - \bar x_n)}{(n-1)(d+1) S^2_n},
     \\
\label{eq:beta_k}
  \hat \beta_k^{(n)} &=& \bar y_{kn} - \bar x_n \hat s^{(n)}, \qquad
  k=0,\ldots,d,
\end{eqnarray}
where $\bar{x}_n := \frac{1}{n}\sum_{j=1}^n x_j$, $S_n^2 =
\frac{1}{n-1}\sum_{j=1}^n (x_j-\bar x_n)^2$ and $\bar y_{kn}  :=
\frac{1}{n}\sum_{j=1}^n y_{kj}$. We propose $\hat s^{(n)}$ as an
estimator of the fractal dimension $s$ and
$\widehat{|\sC_k(F)|}=\exp(\hat \beta_k^{(n)})$ as an estimator of
$|\sC_k(F)|$, $k=0,\ldots,d$.

\begin{remark} \label{rem:J}
If for some indices $k\in\{0,\ldots,d\}$ assumption
(A2) is not satisfied, one can simply abandon
these indices and consider the least-squares estimators based on
the remaining data. In general, for a subset
$J\subseteq\{0,\ldots,d\}$ of indices, one can consider the
estimators
\begin{align}
\label{eq:sJ}
\hat s^{(J,n)} &= \dfrac{\sum_{k\in J}  \sum_{j=1}^n y_{kj} 
    (x_j - \bar x_n)}{(n-1)|J| S^2_n},
     \\
\label{eq:beta_kJ}
  \hat \beta_k^{(J,n)} &= \bar y_{kn} - \bar x_n \hat s^{(J,n)}, \qquad k\in J\,,
\end{align}
which minimize the sum
\begin{equation}
\label{eq:S^2J}
\frac{1}{|J|n} \sum_{k\in J} \sum_{j=1}^n (y_{kj}-(\beta_k+  s
x_j))^2.
\end{equation}
Here $|J|$ denotes the cardinality of the finite set $J$. If
$J=\{d\}$, i.e.~if only the volume $C_d(F_\eps)$ is considered,
the estimators specialize to those provided by the sausage method.
Hence the sausage method is included as a special case in our
considerations. We use the notation $\hat s^{(\{d\},n)}$ and
$\hat{\mathcal{M}}^{(n)}:=\exp(\hat \beta_d^{(\{d\},n)})$ for the
sausage method estimators of the dimension $s$ and the Minkowski
content ${\sM}(F)=\sC_d(F)$ of the set $F$.
\end{remark}

\paragraph{Second method: Quasi--linear regression.}
We assume now that $F\subset\R^d$ is a set satisfying assumptions (A1) and (A2) but not (A3). Instead, we assume the following:
\begin{itemize}
  \item[(A3')] For some $h\in(0,1)$ and each $k=0,\ldots,d$, one has the asymptotic relation
   \begin{align} \label{A3'}
     \eps^{s-k} C_k(F_\eps) \sim p_k(\eps)\quad \text{ as } \eps\searrow 0
   \end{align}
   where $p_k:(0,\infty)\to \R$ is a bounded, (multiplicatively) periodic function with period $h$.
\end{itemize}
This assumption is on the one hand motivated by the known results
for arithmetic self-similar sets, where the existence of such a
periodic function was shown. On the other hand, condition (A3')
ensures the existence of the average fractal curvatures as defined
in \eqref{eq:avg-frac-curv2}. Indeed, $\asC_k(F)$ is given in
terms of the function $p_k$ by \eqref{eq:Int_curv-formula}.

We point out that assumption (A3') is strictly weaker than (A3). For any set $F$ satisfying (A3), the relations \eqref{A3'} hold for the constant
 functions $p_k\equiv\sC_k(F)$ (and some arbitrary $h>0$). This will allow in particular to apply the second method also when the  (non-averaged)
fractal curvatures exist.
Indeed, in a way, the first method can be viewed as a special case of the second method described below.

For a set of $\eps$-values $\eps_1,\ldots,\eps_n$, recall the
notation $x_j = -\log \eps_j$ and $y_{kj} = \log \left(
\eps_j^{-k} |C_k(F_{\eps_j})|\right)$, $k=0,\ldots,d$, from
\eqref{eq:def-xj} and \eqref{eq:def-variables}. The relation
\eqref{A3'}  suggests that if the radii $\eps_j$ are sufficiently
small, then the points $(y_{0j},y_{1j},\ldots, y_{dj})$, $j =
1,\ldots,n$ will lie close to the graph of a function which is the
sum of a linear and a periodic function (see
Figure~\ref{fig:3d-plot}):
\begin{align*} 
\begin{pmatrix}
  y_{0 j} \\ y_{1 j} \\ \vdots \\ y_{d j}
\end{pmatrix}
\approx
sx_j
\begin{pmatrix}
   1 \\ 1 \\ \vdots \\  1
 \end{pmatrix}
 +
\begin{pmatrix}
  \log |p_0(\eps_j)| \\ \log |p_{1}(\eps_j)| \\ \vdots \\ \log |p_d(\eps_j)|
 \end{pmatrix},
\qquad  j=1,\ldots, n.
\end{align*}

For $k=0,\ldots,d$, let the functions $g_k:\R\to (0,\infty)$ and $f_k:\R\to \R$ be given by
\begin{align}\label{eq:f_k}
g_k(x)=|p_k(e^{-x})| \quad \text{ and } \quad f_k(x)=\log
g_k(x)-\beta_k,
\end{align}
where
\begin{align*}
\beta_k:=\frac 1{h_0}\int_0^{h_0} \log
g_k(x) \, dx.
\end{align*}
Note that the multiplicative periodicity of $p_k$ (with period $h=e^{-h_0}$) implies that $f_k$ (as well as $g_k$) is additive periodic with
period $h_0>0$.
The reason to subtract $\beta_k$ in \eqref{eq:f_k} is to center $f_k$, that is, to have $\int_0^{h_0} f_k(x) dx=0$.
Observe that, in case assumption (A3) is satisfied, that is, when $p_k$ is a constant function, $\beta_k$ has the same meaning as before in the first method.

It is plausible to expect the following regression structure
\begin{equation}
\label{eq:regr_final}
 y_{ki}=\beta_k+s \cdot x_i + f_k(x_i)+ \delta_{ki}, \qquad i=1,\ldots, n,
\end{equation}
where $T_k(x):=\beta_k+s \cdot x$ is the \emph{polynomial part}
and $f_k(x)$ is the \emph{seasonal part} of the above time series,
whereas the errors $\delta_{kj}$ are assumed to be correlated
random variables  with $\E \delta_{kj}=0$ and $\var \, \delta_{kj}
\in (0,\infty)$ for $k=0,\ldots, d$ and $j=1,\ldots,n$.

The seasonal part $f_k$ is assumed to be the finite Fourier series
\begin{align*}
f_k(x)=\sum\limits_{j=1}^m \left( \tilde{a}_{kj}\cos \left(2\pi
jx/h_0 \right) + \tilde{b}_{kj}\sin \left(2\pi jx/h_0 \right)
\right)
\end{align*}
for some $m\in\N$. Standard operations with trigonometric
functions allow to write $f_k$ in the form
$$
f_k(x)=\sum\limits_{j=1}^m  b_{kj} \cos \left(\mu_j x +
\varphi_{kj} \right)
$$
with $\mu_j=2\pi j/h_0$ and some $b_{kj}, \varphi_{kj}\in \R$.

First we assume that the number $m$ and the period $h_0$
(and thus all $\mu_j$) are known. Standard methods of time series
analysis (cf.\ e.g.\ \cite[Chapter 9]{Loebus}) can be used to
design the least squares estimators $\hat{\beta}_{k}^{(n)}$,
$\hat{s}_{k}^{(n)}$, $\hat{b}_{kj}^{(n)}$,
$\hat{\varphi}_{kj}^{(n)}$ of $\beta_k$, $s$, $b_{kj}$,
$\varphi_{kj}$
 so that
\begin{equation}\label{eq:EstFm}
 \hat{f}_{k}^{(n)}(x)=\sum\limits_{j=1}^m  \hat{b}_{kj}^{(n)}
\cos \left(\mu_{j} x + \hat{\varphi}_{kj}^{(n)} \right)
\end{equation}
is an estimator of $f_k(x)$.  Namely, regression
\eqref{eq:regr_final} is interpreted as a linear regression with
respect to the parameters $\beta_k$, $s$, $b_{kj} \cos
\varphi_{kj}$, $b_{kj} \sin \varphi_{kj}$, and those are estimated
by ordinary least squares. By the relation
\eqref{eq:Int_curv-formula} and
assumption (A2), 
we have
$$
|\overline{\sC}_k(F)|= \frac {\exp\{\beta_k\}}{h_0}
  \int_0^{h_0} \exp\{ f_k(x)\}\, {d} x,
$$
which allows for the following estimator of (the absolute value of) the $k$--th
fractal curvature
\begin{equation}\label{eq:EstFracCurv}
 |\widehat \sC_{k}^{(n)}(F)|= \frac {\exp\{\hat{\beta}_{k}^{(n)}\}}{h_0}
  \int_0^{h_0} \exp\{ \hat{f}_{k}^{(n)}(x)\}\, {d} x, \quad k=0,\ldots, d.
\end{equation}

In most situations, the numbers $m$ and $h_0$ will be unknown.
They may be estimated as follows. First, the coefficients of the
polynomial part are estimated by means of ordinary least squares
from the linear regression equation \eqref{eq:regr_final} where
$f_k(x_j)+ \delta_{kj}$ are interpreted as new random errors.
Then, the estimated polynomial part is subtracted from the
variables $y_{kj}$, and  the period $h_0$ (and hence frequences
$\mu_j$) are estimated by means of the periodogram of the
resulting time series $\{ \tilde{y}_{kj}, \, j\in\Z \}$ without
trend, see \cite[Section 9.1]{Loebus} and Figure \ref{fig:perio}.
Namely, if $\{ \delta_{ki} \}$ is a stationary ARMA process then
the periodogram
$$I_n(t)=\frac{1}{2\pi n }\left| \sum\limits_{j=1}^n e^{-ijt} \tilde{y}_{kj}  \right|^2, \quad t\in [-\pi,\pi] $$
behaves on average as $\E I_n(t)= O(n)$ as $n\to\infty$ for
$t=\pm\mu_j$ and $t=0$ whereas it is bounded for all other $t$,
cf. \cite[relation (1.5), p.~204]{Loebus}. A more precise result
can be found in \cite[Theorem 5, p.~54]{QuiHan01} for stationary
mixing sequence $\{ \delta_{ki} \}$:
$$
\sup_{t\in J_n} I_n(t)=O(\log\log n) \quad \mbox{a.s.}
$$
where $J_n=(- \log^a n/(2n) \pm\mu_j, \log^a n/(2n) \pm\mu_j )$
for some $a\ge 0$ and any $j=1,\ldots, m$.
 Thus, the positions ($t$--values)
of high peaks of $I_n$ yield estimators for $\mu_j$ (and hence
$h_0$).

For the case of stationary regression errors $\{ \delta_{ki}\}$,
the book \cite[p.~53-54]{QuiHan01} proposes an estimate
\begin{equation}\label{estim_h0}
\widehat{h}_0=2\pi/\widehat{\mu}_1 \quad \mbox{  with } \quad
\widehat{\mu}_1=\mbox{argmax}_{\mu} \sum\limits_{j=1}^m I_n(j\mu).
\end{equation}
We refer the reader to \cite{QuiHan01} for a discussion and
comparison of other estimation methods of the frequency $\mu_1$.

Then the estimated values of $h_0$ and $\mu_j$ should be plugged
into the formula \eqref{eq:EstFracCurv}. The value of $m$ should
be taken reasonably large. Its order can be estimated counting the
number $l$ of high peaks of $I_n$ and setting
$\widehat{m}=\lfloor(l-1)/2\rfloor$ where $\lfloor b \rfloor$
denotes the integral part of a real number $b$, cf.
\cite[p.~205]{Loebus}.

\begin{figure}[ht]
  \centering
\includegraphics[width = 0.8\textwidth]{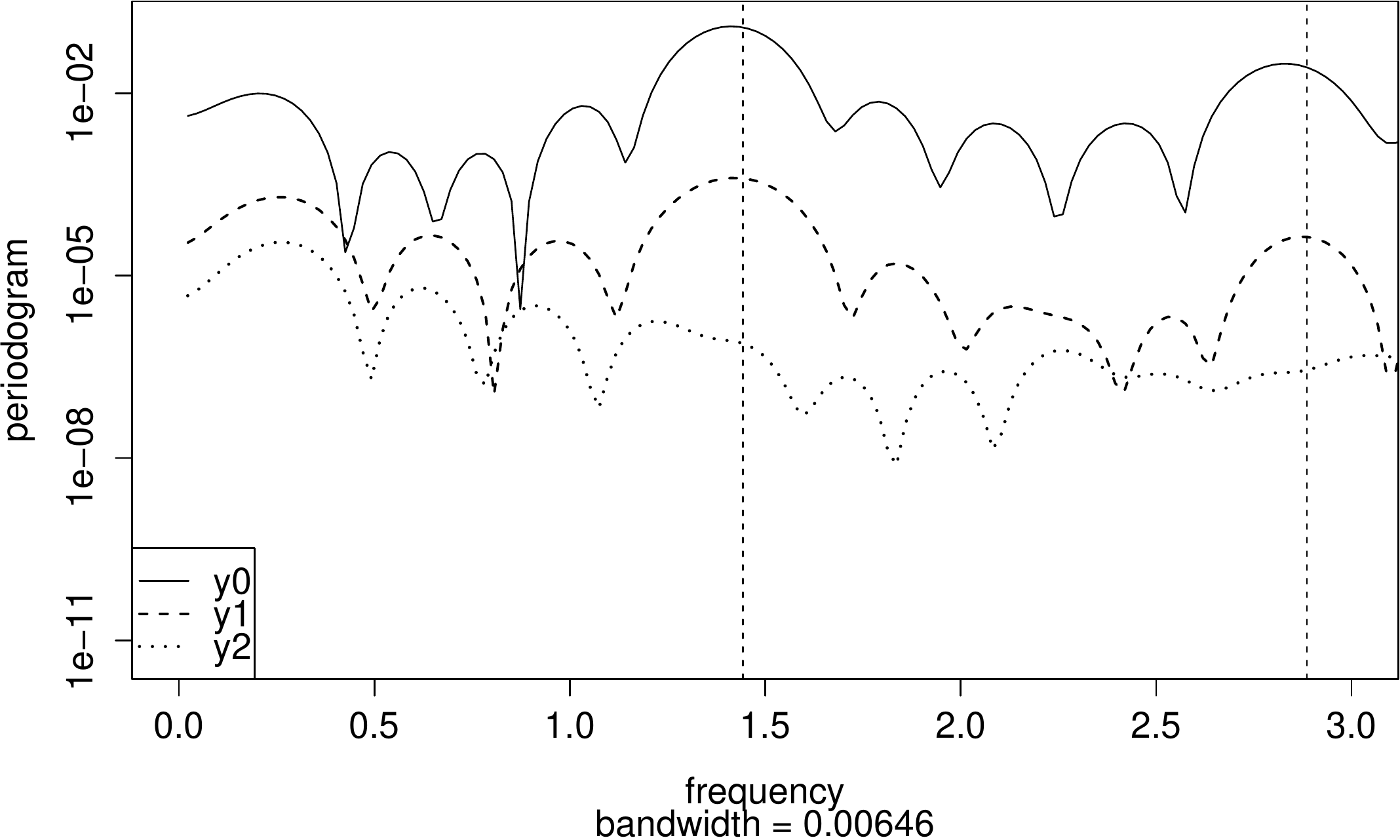}
\caption[Periodogram] {The periodogram of the centered and
rescaled data $(y_0,y_1,y_2)$ from the Sierpi\'{n}ski gasket. The
data is $\log(2)$-periodic, hence the maxima occur at multiples of
$1/\log(2)$ (vertical dashed lines). The periodogram was
calculated using the \texttt R function \texttt{spec.pgram()},
with option \texttt{pad=10}. \label{fig:perio}}
\end{figure}

The procedure of estimating the dimension $s$ and the (absolute
values of) average fractal curvatures
 $\overline \sC_k(F)$ from a given binary image $\tilde F$ of a fractal $F$ now runs as follows:
Dilate  $\tilde F$ by a ball $B_\eps(0)$ for the given set of
radii $\eps\in\{\eps_1,\ldots, \eps_n \}$ and measure the
corresponding intrinsic volumes $|C_k(\tilde F_{\eps})|$,
$k=0,\ldots,d$. Calculate the point cloud
$\{(y_{0j},y_{1j},\ldots, y_{d j})\}_{j=1}^n$ and apply the
regression \eqref{eq:regr_final} either to the data sets
$\{y_{kj}\}_{j=1}^n$ separately  for each $k=0,\ldots, d$ or to
the whole cloud as described above. This results in the estimators
$\hat{s}_{k}^{(n)}$, $|\widehat \sC_{k}^{(n)}(F)|$, $k=0,\ldots,
d$ in the first case and $\tilde{s}^{(n)}$, $|\widetilde
\sC_{k}^{(n)}(F)|$, $k=0,\ldots, d$ in the second case of
simultaneous regression.
\begin{remark}
Using $k$ separate regressions for estimating $s$ and
 $|\overline \sC_{k}(F)|$, $k=0,\ldots, d$,
 notice that the fractal dimension $s$ of $F$ does
not depend on the order $k$ of a considered curvature measure.
However, the estimator $\hat{s}_{k}^{(n)}$ depends on
$k=0,\ldots,d$ as a solution of \eqref{eq:regr_final}. Hence, the
estimation procedure for $s$ can be made more robust by setting $\hat{s}^{(n)}$ to be the empirical median of the sample
$\{ \hat{s}_{k}^{(n)}, \; k=0,\ldots,d \}$.
\end{remark}

The first method can be seen as a special case of the second one
when the seasonal part is assumed to be zero.

\subsection{Preliminary results from linear regression}

The main cause for the inaccuracy of the estimators of Section
\ref{subsec:estimators} is that a large amount of information is
lost in the digitization procedure. The dilation
radius can not be taken arbitrarily small in practice, which would
be necessary for the calculation of a limit. The resolution of the
digitized set determines a lower bound for the dilation radii. However, if we assume that the radii
can be chosen arbitrarily small, i.e., if the resolution increases
to infinity,  then the weak consistency (together with rates of convergence for certain choices of the sequence of radii) of the above estimators can be shown under some rather mild assumptions on the covariance structure of the (random) discretization and computation errors.
Here the choice of the radii can not be completely arbitrary. For simplicity, we only consider monotone sequences. In the first method,
some bound on the speed of decay of the radii will be sufficient, while in the second method we only allow for radii forming an arithmetic sequence on the logarithmic scale.

\paragraph{Weak consistency in the linear regression.}

Consider the classical multivariate linear regression of full rank, i.e.
\begin{align} \label{eq:Regression}
y_l=X_l\beta +\delta_l
\end{align}
with $y_l$ and $\delta_l$ being random vectors of dimension $l\in\N$, $X_l$ the deterministic regression matrix of size $l \times q$
and rank $q\le l$, and $\beta$ being the $q$--dimensional vector of regression parameters.
Note that ${\rm rank}(X_l)=q$ implies that the $q\times q$ matrix $X_l^\top X_l$ is positive definite and has thus strictly positive eigenvalues, which we denote by $\lambda_1,\ldots, \lambda_q$.

Let
$\hat\beta^{(l)}$ be the
least-squares estimator of $\beta$:
\begin{align} \label{eq:LSEst}
\hat\beta^{(l)}=\left( X_l^\top X_l\right)^{-1}X_l^\top y_l \,.
\end{align}
Assume that the coordinates  $\delta_{lj}$, $j=1,\ldots,l$ of $\delta_l$ are a sequence of
(dependent) random variables with positive finite variance
satisfying $\E(\delta_{lj})=0$. Assume that the covariance matrices
$Q_l:=\E(\delta_l \delta_l^\top)$ satisfy
\begin{equation}\label{eq:Ql-ass}
0<\inf_{l\in\N}\nu_{\min}(Q_l)\le
\sup_{l\in\N}\nu_{\max}(Q_l)=:\nu^* <\infty,
\end{equation}
where $\nu_{\min}(Q_l)$ and $\nu_{\max}(Q_l)$ are the
smallest and the largest eigenvalues of $Q_l$, respectively.

\begin{lemma}({\cite[Theorem~3.1]{drygas76}}) \label{lemmDrygas}
 The estimator $\hat\beta^{(l)}$ is weakly
consistent for $\beta$  if and only if $\lambda_{\min}(X_l^\top X_l):=\min_{i=1,\ldots,q} \lambda_i \to
\infty$ as $l\to \infty.$
\end{lemma}
The following result provides an estimate for the rate of convergence of $\hat\beta^{(l)}$ to
$\beta$. We write $\tr(A)$ for the trace of a matrix $A$.
\begin{lemma}\label{lem:weakCons_linReg}
For each $\eps>0$,
\begin{equation}\label{eq:estWeakCons}
 P(|\hat\beta^{(l)}-\beta|>\eps)\le \frac{\nu^*}{\eps^2} \frac{\tr(X_l^\top X_l)}{\left(\lambda_{\min}(X_l^\top X_l)\right)^2}\,. 
 \end{equation}
\end{lemma}
\begin{proof}
It is easy to see that  $\hat\beta^{(l)}-\beta=\left( X_l^\top X_l\right)^{-1}X_l^\top \delta_l$ and the covariance matrix of $\hat\beta^{(l)}-\beta$
is given by $\mbox{Cov}(\hat\beta^{(l)}-\beta)=\left( X_l^\top X_l\right)^{-1}X_l^\top Q_l X_l \left( X_l^\top X_l\right)^{-1} $.
Using the  Markov inequality, we get
$$
 P(|\hat\beta^{(l)}-\beta|>\eps)\le \tr \left( \mbox{Cov}(\hat\beta^{(l)}-\beta) \right) /\eps^2\,.
$$
Since the matrices $Q_l$ and $X_l^TX_l$ are symmetric, they can be diagonalized, that is, there exist orthogonal matrices $C, B$ and diagonal matrices
$$
\Lambda=\mbox{diag}(\lambda_1,\ldots,\lambda_q) \quad \mbox{ and } \quad N=\mbox{diag}(\nu_1,\ldots,\nu_l)
$$
such that $X_l^\top X_l=C^\top \Lambda C$ and  $Q_l=B^\top N B$, respectively, where the eigenvalues $\lambda_i$ of  $X_l^TX_l$ and $\nu_j$ of $Q_l$ are strictly positive, since $X_l^TX_l$ has full rank and due to assumption \eqref{eq:Ql-ass}, respectively.
Note that  $\left( X_l^\top X_l\right)^{-1}=C^\top \Lambda^{-1} C$ with $\Lambda^{-1}=\mbox{diag}(1/\lambda_1,\ldots,1/\lambda_q)$.
Using the cyclic commutativity property of the trace one gets
after standard calculations that
\begin{align*}
\tr \big( &\mbox{Cov}(\hat\beta^{(l)}-\beta) \big)=
\tr \left(\Lambda^{-2} C X_l^\top B^\top N B X_l C^\top
\right) \le  \max_{i=1,\ldots,q} \lambda_i^{-2}\cdot  \tr
\left( C X_l^\top
B^\top N B X_l C^\top \right)  \\
 &=  \left( \lambda_{\min}(X_l^\top X_l)\right)^{-2}  \tr \left(
X_l X_l^\top B^\top N B  \right)\le \left( \lambda_{\min}(X_l^\top X_l)\right)^{-2} \max_{i=1,\ldots,q} \nu_i \cdot \tr
\left( B X_l X_l^\top B^\top   \right)\\
 &\le \left(
\lambda_{\min}(X_l^\top X_l)\right)^{-2} \nu^* \cdot \tr
\left( X_l  X_l^\top \right)=  \left( \lambda_{\min}(X_l^\top X_l)\right)^{-2} \nu^* \tr
\left( X_l^\top  X_l \right).
\end{align*}
In the derivation we have used in particular that the matrices $C X_l^\top
B^\top N B X_l C^\top$ and $B X_l X_l^\top B^\top $ are covariance matrices with nonnegative entries on the diagonal. This completes the proof of \eqref{eq:estWeakCons}.
\end{proof}

\paragraph{Asymptotic normality.}

Imposing additional assumptions on the dependence structure of
regression errors in  \eqref{eq:Regression}, one can prove the
asymptotic normality of the least squares estimator
$\hat\beta^{(l)}$. For simplicity, we do it for
$\delta_l=\sqrt{Q_l} \gamma_l$, where $\sqrt{Q_l}$ is the square
root of the symmetric positive definite matrix $Q_l \in \mathbb
R^{l \times l}$ and $\gamma_l=(\gamma_{l1},\ldots,
\gamma_{ll})^\top$ is a random vector with iid coordinates
$\gamma_{lj}$, $\E (\gamma_{lj})=0$, $\E (\gamma_{lj}^2)=1$ for
all $j=1,\ldots,l$. This corresponds to the case when for each $l$
$(\delta_{lj})$ is a linear process with a finite range of
dependence.

\begin{theorem} \label{thm:ass-normality}
Under the assumptions on $Q_l$ in the paragraph above, it holds
\begin{equation}\label{eq:assnorm}
\frac{t^\top (\hat\beta^{(l)}-\beta)}{\sqrt{t^\top  \left(
X_l^\top X_l\right)^{-1}X_l^\top Q_l X_l \left( X_l^\top
X_l\right)^{-1} t}} \tod Z\sim N(0,1),\qquad l\to\infty
\end{equation}
for all $t\in\R^q\setminus\{0\}$ such that
\begin{equation}\label{eq:assnorm_cond}
\frac{\| A_l^\top t \|_2}{\| A_l^\top t \|_\infty}
\to\infty,\qquad l\to\infty,
\end{equation}
where $ A_l=\left( X_l^\top X_l\right)^{-1} X_l^\top \sqrt{Q_l} \in \mathbb R^{q \times l} $,
$\|\cdot \|_2$ and $\|\cdot \|_\infty$ are the Euclidean and
the maximum norm in $\R^l$, respectively,  and $\hat\beta^{(l)}$ is
the estimator \eqref{eq:LSEst}.
\end{theorem}
\begin{proof}
The result follows from the central limit theorem with an application
of the Lindeberg condition.
Define the vector
$$
b_{l}=\frac{A_l^\top t }{\|A_l^\top t  \|_2 } \in \mathbb R^{l}.
$$
and write $ b_{lj} $ for its $ j $-th coordinate.
Let $ Y_{lj} := b_{lj}\gamma_{lj} $.
Since
$$
\| A_l^\top t\|_2^2=t^\top  \left( X_l^\top
X_l\right)^{-1}X_l^\top Q_l X_l \left( X_l^\top X_l\right)^{-1} t,
$$
we notice that $ Y_l := \sum_{j=1}^l Y_{lj} $ equals the left hand side of \eqref{eq:assnorm}.
The doubly indexed sequence $ \{Y_{lj}: 1 \le j \le l \} $ satisfies $\E\,Y_{lj}=0 $,
$\E\,Y_{lj}^2= b_{lj}^2$, $\sum_{j=1}^l b_{lj}^2=1$ for all
$j=1,\ldots,l$ and $l\in\N$. It also satisfies the Lindeberg condition
(see e.g.\ \cite{Feller68}), which can be written in the following form: for each $\eps > 0$,
\begin{equation*}
\lim\limits_{l \to \infty} \sum\limits_{j=1}^l \E\,\left( Y_{lj}^2 \Ind(Y_{lj}^2>\eps^2)\right) = 0.
\end{equation*}
Indeed, employing condition \eqref{eq:assnorm_cond}, for any $\eps>0$ we have
\begin{multline*}
\sum\limits_{j=1}^l \E\,\left( Y_{lj}^2 \Ind(Y_{lj}^2>\eps^2)
\right)= \sum\limits_{j=1}^l b_{lj}^2 \E\,\left( \gamma_{11}^2
\Ind(\gamma_{11}^2>\eps^2/b_{lj}^2) \right)\\
\le \E\,\left( \gamma_{11}^2 \Ind(\gamma_{11}^2>\eps^2  \|A_l^\top
t  \|_2^2/ \|A_l^\top t  \|_\infty^2) \right) \cdot
\sum\limits_{j=1}^l b_{lj}^2 = \E\,\left( \gamma_{11}^2
\Ind(\gamma_{11}^2>\eps^2  \| A_l^\top t\|_2^2/ \|A_l^\top t
\|_\infty^2) \right)\to 0
\end{multline*}
as $l\to\infty$ due to the integrability of $\gamma_{11}^2$.
\end{proof}
The following statement gives a sufficient condition for
\eqref{eq:assnorm_cond}.
\begin{corollary} \label{cor:assnorm_cond2}
If $Q_l=\sigma^2 I_l$ for some $ \sigma > 0 $, where $I_l$ is the
identity matrix, then condition \eqref{eq:assnorm_cond} in
Theorem~\ref{thm:ass-normality} is satisfied for all $t\neq 0$
provided that
\begin{equation}\label{eq:ass_norm_cond_iid}
\lambda_{\min}(X_l^\top X_l) / \| X_l \|_\infty^2\to \infty,
\qquad l\to\infty,
\end{equation}
where $\| X_l \|_\infty = \max_{j\in{1.\ldots,l}} \sum_{k=1}^q
|(X_l)_{jk}|$ is the maximum absolute row sum of the matrix $X_l$
and $\lambda_{\min}(X_l^\top X_l)$ denotes (as before) the
smallest eigenvalue of $X_l^\top X_l$.
\end{corollary}
\begin{proof}
Suppose that $ Q_l = \sigma^2 I_l$.  Using the diagonalisation of $(X_l^\top X_l)^{-1}=C^T\Lambda^{-1}C$ from the proof of Lemma
\ref{lem:weakCons_linReg}, the matrix $A_l A_l^\top$ (defined in Theorem~\ref{thm:ass-normality}) can be represented by
$A_l A_l^\top =  \sigma^2 (X_l^\top X_l)^{-1}= \sigma^2 C^\top
\Lambda^{-1} C$. This yields
\begin{equation*}
\|A_l^\top t \|_2^2 =  t^\top A_l A_l^\top t = \sigma^2 t^\top C^\top
\Lambda^{-1} C t = \sigma^2 \sum\limits_{i=1}^{q}
\frac{b_i^2}{\lambda_j}
\end{equation*}
where $ b =(b_1,\ldots, b_q)^\top := C t $. Moreover, we have
$ A_l^\top = \sigma X_l (X_l^\top X_l)^{-1} = \sigma X_l C^\top
\Lambda^{-1} C, $ and
hence
\begin{align*}
\|A_l^\top t \|_\infty^2 &\le \sigma^2 \|X_l\|_\infty^2 \| C^\top
\Lambda^{-1} C t\|_\infty^2  \le c_1 \sigma^2
\|X_l\|_\infty^2 \| C^\top \Lambda^{-1} C t\|_2^2 \\ & \le c_1
\sigma^2 \|X_l\|_\infty^2 b^\top \Lambda^{-2} b = c_1 \sigma^2
\|X_l\|_\infty^2 \sum\limits_{i=1}^{q}\frac{b_i^2}{\lambda_i^2}
\end{align*}
for some constant $ c_1 > 0 $, using the fact that all norms in
finite-dimensional spaces are equivalent. After some more algebra,
we arrive at the desired estimate:
\begin{align*}
\frac{\|A_l^\top t\|_2^2}{\|A_l^\top t\|_\infty^2} &\ge
\frac{\sum_{i=1}^{q} b_i^2 / \lambda_i} {c_1
\|X_l\|_\infty^2\sum_{i=1}^{q} b_i^2 / \lambda_i^2} =
\frac{\sum_{i=1}^q b_i^2 \left( \prod_{j\neq i}\lambda_j \right) /
(\lambda_1 \cdot \ldots \cdot \lambda_q)}
{c_1 \|X_l\|_\infty^2 \sum_{i=1}^q b_i^2 \left( \prod_{j\neq i}\lambda_j^2 \right)  / (\lambda_1 \cdot \ldots \cdot \lambda_q)^2} \\
&= \frac{\lambda_1 \cdot \ldots \cdot \lambda_q}{c_1
\|X_l\|_\infty^2} \frac{\sum_{i=1}^q b_i^2 \prod_{j\neq
i}\lambda_j} {\sum_{i=1}^q b_i^2 \prod_{j\neq i}\lambda_j^2} =
\frac{1}{c_1 \|X_l\|_\infty^2} \frac{\sum_{i=1}^q \lambda_i
\prod_{j\neq i}\lambda_j^2 b_i^2} {\sum_{i=1}^q \prod_{j\neq
i}\lambda_j^2 b_i^2} \ge \frac{\min_{i=1,\ldots, q} \lambda_i}{c_1
\|X_l \|_\infty^2}.
\end{align*}

\end{proof}

\begin{remark}
If $Q_l$ is known or can be consistently estimated from $m$
independent copies of $y_l$ by $\hat Q_{l,m} \tow Q_l$ as
$m\to\infty$ for each $l$, then the above theorem (together with a
Slutsky argument) can be used in a standard way (see e.g. {\rm
\cite[Section 6.3.2, p.~398]{BickelDoksum01}}) to construct an
asymptotic confidence  interval for the coordinates of $\beta_l$
as well as a large sample Wald's test of the hypothesis $H_0$:
$\beta_{j}=\beta_{j,0}$ vs. $H_1$: $\beta_{j}\neq \beta_{j,0}$ for
a fixed $\beta_{j,0}$ and $j=0,\ldots, q$.
\end{remark}

\subsection{Asymptotics of the first method}

Let $\delta_{kj}$, $k=0,\ldots,d$, $j\in\N$ be a sequence of
(dependent) random variables with positive finite variance
satisfying $\E(\delta_{kj})=0$ and assumption \eqref{eq:Ql-ass}.  Let $n\ge 2$.
Since the measured
intrinsic volumes show an almost periodic oscillatory behaviour
with a ``period'' much larger than the step width of the radii,
assuming no correlations would not be very reasonable.
We suppose that the random
variables $y_{kj}$ can be represented in the form
\begin{align*}
y_{kj}=\beta_k+ x_j s +\delta_{kj}, \qquad k=0,\ldots,d, \quad
j\in\N\,.
\end{align*}
Restricting the consideration to the first $n$ observations (i.e.\ to the data derived from the first $n$ radii),
this is more conveniently expressed in matrix form \eqref{eq:Regression} with $l=n(d+1)$, $q=d+2$,
\begin{align*}
y_l:=\begin{pmatrix} y_{01}\\ \vdots\\ y_{d1}\\[5mm] \vdots \\[5mm] y_{0n}\\ \vdots \\
y_{dn}\end{pmatrix},\qquad
X_l:=\begin{pmatrix} 1&0&\ldots& 0 & x_1\\
0&\ddots&\ddots&\vdots&\vdots\\
\vdots&\ddots&\ddots&0&\vdots\\
0&\ldots&0&1&x_1\\[5mm]
&&\vdots&\\[5mm]
1&0&\ldots& 0 & x_n\\
0&\ddots&\ddots&\vdots&\vdots\\
\vdots&\ddots&\ddots&0&\vdots\\
0&\ldots&0&1&x_n
\end{pmatrix}\,,
\qquad  \delta_l:=\begin{pmatrix} \delta_{0 1}\\ \vdots\\ \delta_{d 1}\\[5mm] \vdots\\[5mm] \delta_{0n}\\ \vdots\\ \delta_{dn}\end{pmatrix},
\end{align*}
and $\beta:=(\beta_0,\ldots,\beta_d, s)^\top$. The vector $\beta$ contains the
total curvatures and fractal dimension to be estimated. In the sequel, we adopt the notation slightly and write $X_n$ (instead of $X_l$) for the regression matrix based on the first $n$ radii (of size $n(d+1)\times d+2$) and similarly $Q_n$ for the $n(d+1)\times n(d+1)$ matrix describing the covariance structure of the errors $\delta_{lj}$.
Similarly, we will write
$\hat\beta^{(n)}:=(\hat \beta_0^{(n)},\ldots,\hat \beta_d^{(n)},
\hat s^{(n)})^\top$ for the corresponding
least-squares estimator \eqref{eq:LSEst} of $\beta$ based on the first $n$ radii.

\begin{theorem} \label{thm:weak-consistency}
Let $F\subset\R^d$ be a set satisfying the assumptions (A1)-(A3).
Let $\eps_1>\eps_2>\ldots>0$ be a decreasing sequence of radii
satisfying the condition
\begin{equation}\label{eq:Radii}
\frac{\bar x_n^2}{\widetilde{S}_n^2} = O(n^\mu) \qquad \text{ as } n\to\infty,
\end{equation}
for some $\mu\in[0,1)$, where $x_j=-\log \eps_j$, $j\in\N$,  $\bar x_n=1/n \sum
\limits_{i=1}^n x_i$, and
$\widetilde{S}_n^2=1/n\sum\limits_{i=1}^n (x_i-\bar x_n)^2$.
 Suppose that the covariance matrices $Q_n$ of the
errors satisfy the assumption (\ref{eq:Ql-ass}). Then with the
notation above, the sequence $ \hat\beta^{(n)}$ of least-squares
estimators of $\beta$ is weakly consistent, i.e., for each
$\eps>0$,
\begin{align*}
   P(|\hat\beta^{(n)}-\beta|>\eps)\to 0 \quad \text{ as } n\to\infty.
\end{align*}
\end{theorem}
\begin{proof}
By Lemma \ref{lemmDrygas}, $(\hat\beta^{(n)})$ is a consistent
sequence of estimators of $\beta$ if and only if
\begin{align*}
\lim_{n\to\infty} \lambda_{\min}^*(X_n^\top X_n)=\infty\,,
\end{align*}
where $\lambda_{\min}^*(X_n^\top X_n)$ denotes the smallest
positive eigenvalue of $X_n^\top X_n$. We have
\begin{align*}
 X_n^\top X_n =
\begin{pmatrix}
n & 0 & \ldots & 0 & n \bar x_n\\
0 &\ddots &\ddots&\vdots    & \vdots\\
\vdots & \ddots &\ddots&    0& \vdots\\
0 & \ldots & 0& n   & n \bar x_n\\
n \bar x_n& \ldots && n \bar x_n & v_n
\end{pmatrix} \in\R^{(d+2)\times(d+2)}
\end{align*}
with $v_n:=(d+1)\sum_{j=1}^n x_j^2$. Since $\mbox{\rm
rank}(X_n)=d+2$, the symmetric matrix $X_n^\top X_n$ is positive
definite implying that all its eigenvalues are positive. Since
\begin{eqnarray*}
\lefteqn{\det(X_n^\top X_n-\lambda I)}\\
& = & \begin{vmatrix} n-\lambda &0 & \ldots& \ldots& 0 &n \bar x_n\\
-(n-\lambda) & \ddots &\ddots&&\vdots& 0\\
\vdots & 0 & \ddots& \ddots &\vdots& \vdots\\
\vdots & \vdots &\ddots & \ddots& 0 &\vdots\\
-(n-\lambda) & 0 & \ldots & 0 & n-\lambda & 0\\
n\bar x_n & \ldots &&\ldots & n\bar x_n & v_n-\lambda
\end{vmatrix}
=
\begin{vmatrix}
n-\lambda & 0 & \ldots & 0 & n \bar x_n\\
0 &\ddots &\ddots&\vdots    & 0\\
\vdots & \ddots &\ddots&    0& \vdots\\
0 & \ldots & 0& n-\lambda   & 0\\
(d+1) n \bar x_n& n \bar x_n & \ldots & n \bar x_n & v_n-\lambda
\end{vmatrix}\\
&=& (n-\lambda)^{(d+1)} (v_n-\lambda) - (d+1) (n \bar x_n)^2  (n-\lambda)^d \\
&=& (n-\lambda)^d \left(\lambda^2 - (n+v_n)\lambda +
n^2(d+1)\widetilde{S}_n^2\right)\,,
\end{eqnarray*}
the eigenvalues of
$X_n^\top X_n$ are $\lambda_0^{(n)}=n$ (of multiplicity $d$) and
\begin{align} \label{eq:eigval-XTX}
\lambda_{1/2}^{(n)}= \frac{n+v_n}{2} \pm
\sqrt{\frac{(n+v_n)^2}{4}- n^2(d+1)\widetilde{S}_n^2}\,.
\end{align}
Since, obviously, $\lambda_0^{(n)}\to\infty$ as $n\to\infty$ and
$\lambda_1^{(n)}\ge \lambda_2^{(n)}>0$, it suffices to show that
$\lambda_2^{(n)}\to\infty$ as $n\to \infty$. We have
\begin{align}
2 \lambda_{2}^{(n)}&= n+v_n - \sqrt{(n+v_n)^2- 4n^2(d+1)\widetilde{S}_n^2} \notag \\
 & =  \frac{4n^2(d+1)\widetilde{S}_n^2}{n+v_n+\sqrt{(n+v_n)^2- 4n^2(d+1)\widetilde{S}_n^2}} \label{eq:eigval-2}\\
 &\ge \frac{4n^2(d+1)\widetilde{S}_n^2}{2(n+v_n)}=
\frac{2(d+1)n\widetilde{S}_n^2}{1+v_n/n}\,,\notag
\end{align}
where the inequality is due to the fact that the expression under
the root is non-negative and not larger than $(n+v_n)^2$. Since
$v_n=n(d+1)\widetilde{S}_n^2 + n(d+1)\bar x_n^2$, we obtain
\begin{align}\label{eq:eigval-3}
  \lambda_{2}^{(n)}\ge \frac{(d+1)n}{1/\widetilde{S}_n^2+  (d+1)\left(1+ \bar x_n^2/\widetilde{S}_n^2\right) }\ge \frac{n}{1/\widetilde{S}_n^2+  \left(1+ \bar x_n^2/\widetilde{S}_n^2\right) }\longrightarrow
  \infty
\end{align}
provided that $$ 1/\widetilde{S}_n^2+ \left(1+ \bar
x_n^2/\widetilde{S}_n^2\right)  =1+ \frac{1+\bar
x_n^2}{\widetilde{S}_n^2} = O(n^\mu), \quad \text{ as } n\to\infty\,,
$$
for some $0\le\mu<1$. The last condition is satisfied due to assumption \eqref{eq:Radii}.
\end{proof}

\begin{example}
Condition \eqref{eq:Radii} is satisfied in particular for any
arithmetic sequence of the form $x_j=a_0+a\cdot j$, $j\in\N$ where
$a_0\ge0$ and  $a>0$. Without loss of generality, we demonstrate this for $a_0=0$, $a=1$, that is $x_j=j$, $j\in\N$. In this case we have $\bar x_n=(n+1)/2$ and
$\widetilde{S}_n^2=(n+1)(4n^2-n-3)/12n$, hence $ \bar
x_n^2/\widetilde{S}_n^2=3n(n+1)/(4n^2-n-3)\to 3/4$ as $n\to\infty$. The
condition \eqref{eq:Radii} is satisfied with $\mu=0$. This means
that the estimator $\hat{\beta}^{(n)}$ is weakly consistent for a
sequence of dilation radii $\eps_j=e^{-a_0-a\cdot j}$, $j\in\N$,
$a_0\ge0$, $a>0$.
\end{example}
Recall that the relation $f=\Theta(g)$ for $f,g:\R\to \R$ means that
there exist constants $c_1,c_2>0$ such that $c_1|g(x)| \leq |f(x)|
\leq c_2 |g(x)|$ for all sufficiently large $x$.

\begin{theorem} \label{thm:weak-consistencyRate}
Let $F\subset\R^d$ be a set satisfying
the assumptions (A1)-(A3). Let $\eps_1>\eps_2>\ldots>0$ be a decreasing
sequence of radii. Suppose that the covariance matrices $Q_l$ of the
errors satisfy the assumption (\ref{eq:Ql-ass}).
Then, for any $\eps>0$,
\begin{equation}\label{eq:estWeakConsMeth1}
 P(|\hat\beta^{(n)}-\beta|>\eps)\le \frac{4\nu^*}{\eps^2 (d+1)n
  } \left( 1+ \bar x_n^2 + \widetilde{S}_n^2 \right)\left(1/\widetilde{S}_n^2 + (d+1)\left(
  1+\frac{\bar x_n^2}{ \widetilde{S}_n^2}  \right)\right)^2.
 \end{equation}
If there are constants $\gamma,\mu\geq 0$ such that the sequence of radii satisfies the conditions
\begin{equation}\label{eq:CondRadii}
{\widetilde{S}_n^2} = \Theta(n^{\gamma}),   \quad \text{ and } \quad   \bar x_n=O(n^{\frac{\mu}{2} }),      \quad \text{ as } n\to\infty
\end{equation}
with  $\alpha:=1-\max\{\gamma,\mu\}-2\max\{0,\mu-\gamma\}>0$, then the sequence $\hat \beta^{(n)}$
of least-squares estimators of $\beta$ is weakly consistent with
the rate of convergence
\begin{align}\label{eq:ConvRate}
   P(|\hat\beta^{(n)}-\beta|>\eps)= O(n^{-\alpha})\quad \text{ as } n\to\infty
\end{align}
for any $\eps>0$.
\end{theorem}
\begin{proof}
To show \eqref{eq:estWeakConsMeth1}, we apply Lemma~\ref{lem:weakCons_linReg}, compute $\tr(X_l^\top X_l)$ and estimate $\lambda_{\min}(X_l^\top X_l)$ from below. The trace can be read off directly from the matrix. Using that $v_n=n(d+1)\widetilde{S}_n^2 + n(d+1)\bar x_n^2$, we get
$$
 \tr(X_l^\top X_l)=n(d+1)+v_n=n(d+1)\left( 1+ \bar x_n^2 + \widetilde{S}_n^2 \right)\,.
$$
For the eigenvalues of $X_n^\top X_n$, we claim that
\begin{align} \label{eq:min-eigval-claim}
  \lambda_{\min}(X_n^\top X_n)\geq \frac 12\lambda^{(n)}_2\,,
\end{align}
where $\lambda^{(n)}_2$ is given in \eqref{eq:eigval-XTX}. Indeed, from the proof of Theorem~\ref{thm:weak-consistencyRate}, we have $\lambda_{\min}(X_n^\top X_n)=\min\{n,\lambda^{(n)}_1, \lambda^{(n)}_2\}$ with $\lambda^{(n)}_1\geq \lambda^{(n)}_2$. To prove the claim, it therefore suffices to show that $n\geq \lambda^{(n)}_2/2$ that is $\lambda^{(n)}_2/n \leq 2$. From \eqref{eq:eigval-2} it is easily seen that
$$
\frac{\lambda^{(n)}_2}{n} \leq \frac{2n(d+1)\widetilde{S}_n^2}{v_n}=2 \frac{\sum_{j=1}^n x_j^2-n\bar x_n^2}{\sum_{j=1}^n x_j^2}\le 2,
$$
since the last numerator is clearly positive and smaller than the denominator. This proves \eqref{eq:min-eigval-claim}. To complete the proof of \eqref{eq:estWeakConsMeth1}, it suffices now to combine \eqref{eq:min-eigval-claim} with \eqref{eq:eigval-3} to see that
$$
(\lambda_{\min}(X_n^\top X_n))^{-2}\leq 4 \left(\lambda^{(n)}_2\right)^{-2}\leq
  \frac{4 \left(1/\widetilde{S}_n^2+  (d+1)\left(1+ \bar x_n^2/\widetilde{S}_n^2\right)\right)^2}{(d+1)^2n^2}.
$$
The relation \eqref{eq:ConvRate} follows
easily from \eqref{eq:estWeakConsMeth1} and \eqref{eq:CondRadii}.
 The expression in the first parentheses of the right-hand side in \eqref{eq:estWeakConsMeth1} is bounded from above (up to some constant) by $n^{\max\{\mu,\gamma\}}$, while the expression in the second parentheses is bounded up to a constant by $n^{\max\{0,\mu-\gamma\}}$. 
\end{proof}

\begin{example}
Condition \eqref{eq:CondRadii} is satisfied for
$x_j=O(j^{\delta})$, $\delta\in (0,1/2)$  with $\gamma=\mu=2\delta$,
and $\alpha=1-2\delta$. This means that the estimator
$\hat{\beta}^{(n)}$ is weakly consistent for a sequence of
dilation radii $\eps_j=e^{-c j^{\delta}}$, $j\in\N$, $c>0$ with
the rate of convergence $O\left(n^{-(1-2\delta)}\right)$.

Unfortunately, Lemma~\ref{lem:weakCons_linReg} and Theorem~\ref{thm:weak-consistencyRate} are not strong enough to provide a rate of convergence in the case of an arithmetic sequence $(x_j)$. For $x_j=j, j\in\N$, one has $\tr(X_n^\top X_n)=(d+1)n(6+(n+1)(2n+1))/6=O(n^3)$ as $n\to\infty$, while it can be shown that $\lambda_{\min}(X_n^\top X_n)=\Theta(n)$ as $n\to\infty$, meaning that the right hand side in the estimate \eqref{eq:estWeakCons} still grows linearly as $n\to\infty$.
\end{example}

\begin{corollary} \label{cor:cons_curvatures}
Under the assumptions of  Theorem \ref{thm:weak-consistencyRate},
the estimators $\widehat{|\sC_k(F)|}=\exp(\hat \beta_{k}^{(n)})$,
$k=0,\ldots,d$ of the (absolute values of) the fractal curvatures
are weakly consistent, i.e., for any $\eps>0$
$$ P(\left|\widehat{|\sC_k(F)|}-|\sC_k(F)|\right|>\eps)=O(n^{-\alpha}) \; \mbox{ as } \; n\to\infty,\qquad k=0,\ldots,d.$$
Similarly, the estimator $\hat s^{(n)}$ 
of the dimension $s$ is weakly consistent with the same
convergence rate.
\end{corollary}
\begin{proof} Let $\delta>0$.
According to Taylor's theorem, for each $t$ with $|t|\leq\delta$,
there is a $\xi=\xi(t)\in[-\delta,\delta]$ such that $e^t=1+t+t^2
e^{\xi}/2$.
For $x,y\in\R$ such that $|x-y|\leq\delta$ this gives $e^{x-y}=1+(x-y)+(x-y)^2 e^\xi/2$ and thus $e^x-e^y= e^{y} \left( x-y + (x-y)^2 e^\xi /2\right)$ for some $\xi=\xi(x,y)\in[-\delta,\delta]$. Since $e^{-\delta}\leq e^\xi\leq e^\delta$, we infer that
\begin{equation*}
e^{y} \left( x-y + (x-y)^2 e^{-\delta} /2\right) \leq e^x-e^y\leq e^{y} \left( x-y + (x-y)^2 e^\delta /2 \right)\,
\end{equation*}
and thus
\begin{align}\label{eq:exp}
|e^x-e^y|\leq e^y \max_{s\in\{-\delta,+\delta\}}\left\{\left| x-y + (x-y)^2 e^s /2 \right| \right\}\leq e^y \left| x-y\right| + (x-y)^2 e^{y+\delta} /2  \,.
\end{align}
Now set $x=\hat \beta_k^{(n)}$ and $y= \beta_k $ for brevity.
Using the relation \eqref{eq:exp} we infer that, for any $\eps>0$,
\begin{align*}
P&\left(|e^x-e^y|> \eps  \right)\\
&\le P\left(|e^x-e^y|> \eps,\,|x-y|\le \delta \right) + P\left(|e^x-e^y|> \eps \big|\, |x-y|> \delta \right) P\left(|x-y|> \delta \right) \\
&\leq P\left(e^{y}|x-y|>\eps/2,\,|x-y|\le \delta \right) + P\left(e^{y+\delta}(x-y)^2>\eps ,\,|x-y|\le \delta \right) + P\left(|x-y|> \delta \right)\\
&\leq P\left(|x-y|>e^{-y}\eps/2\right) + P\big(|x-y|>\sqrt{\eps} e^{-(y+\delta)/2} \big) + P\left(|x-y|> \delta \right)\,.
\end{align*}
Now we apply the estimate \eqref{eq:estWeakConsMeth1} to each of the terms in the last sum. Noting that $|\hat \beta^{(n)}-\beta|\geq|\hat \beta_k^{(n)}-\beta_k|$, we obtain, for each $\eps>0$ (and each $\delta>0$),
\begin{align*}
 P\left(|\exp(\hat \beta_k^{(n)})-\exp( \beta_k )|> \eps \right)&\le c_k  \frac{4\nu^*}{(d+1)n} \left( 1+ \bar x_n^2 + \widetilde{S}_n^2 \right)\left(1/\widetilde{S}_n^2 + (d+1)\left(
  1+\frac{\bar x_n^2}{ \widetilde{S}_n^2}  \right)\right)^2.
\end{align*}
where the constant $c_k:=\left(4e^{2\beta_k}\eps^{-2}+e^{\beta_k+\delta}\eps^{-1}+\delta^{-2}\right)$ depends on $\beta_k$ (and the chosen $\delta$) but not $n$. Now the claimed convergence rate follows from condition \eqref{eq:CondRadii} in the same way as in the proof of \eqref{eq:ConvRate} above.
The convergence rate for the dimension estimators is just a reformulation of \eqref{eq:ConvRate} taking into account that $s=\beta_{d+2}$  and $\hat s^{(n)}=\hat \beta_{d+2}^{(n)}$.
\end{proof}

Note that the same consistency results hold for the estimators $\hat{s}^{(J,n)}$ and $\hat{\beta}_k^{(J,n)}$ for any subset $J\subseteq\{0,\ldots,d\}$  
such that assumption (A2) 
is satisfied for all $k\in J$, cf.~Remark~\ref{rem:J}. In particular, it applies to the sausage method. In this case, we can formulate the result in greater generality. The assumption (A1) is not needed (as the volume is always well defined) and (A2) is always satisfied (as the volume is positive). (A3) simplifies to the existence of the Minkowski content of $F$.

\begin{corollary} \label{thm:sausage}
Let $F\subset\R^d$ be a set whose Minkowski content exists and let $\eps_1>\eps_2>\ldots>0$ be a decreasing sequence of
radii. Suppose that the conditions \eqref{eq:Ql-ass} and
\eqref{eq:CondRadii} are satisfied. Then the sausage method estimators
$\hat s^{(\{d\},n)}$ and  $\hat \sM^{(n)}(F)$ are weakly consistent. More precisely,
for each $\eps>0$,
$$
P(|\hat s^{(\{d\},n)}-s|>\eps)=O(n^{-\alpha})\quad  \mbox{ and } \quad
 P(|\hat \sM^{(n)}-{\sM}|>\eps)=O(n^{-\alpha}) \quad \mbox{ as } \; n\to\infty\,,
$$
with $\alpha$ as in Theorem \ref{thm:weak-consistencyRate}.
\end{corollary}

As a last result for the first method, we show the asymptotic
normality of the estimators $\hat \beta^{(n)}$ using
Corollary~\ref{cor:assnorm_cond2}, for which stronger assumptions
on the covariance structure of the errors (no correlation) are
required. However, these assumptions are not realistic since
$\delta_{ki}$ are clearly dependent, see Remark
\ref{remRegrErrors}. A more general correlation structure $Q_n$
would require to verify the condition \eqref{eq:assnorm_cond}
which seems to be quite tedious.

\begin{theorem} \label{thm:assnorm-method1}
Let $F\subset\R^d$ be a set satisfying
the assumptions (A1)-(A3). Assume that $Q_n=\sigma^2 I$ for some $ \sigma > 0 $, where $I$ is the identity matrix. Let $\eps_1>\eps_2>\ldots>0$ be a decreasing
sequence of radii and suppose there are constants $\gamma,\mu\geq 0$ with $\max\{\mu, 2\mu-\gamma\}<1$ such that
\begin{equation}\label{eq:CondRadii2}
{\widetilde{S}_n^2} = \Theta(n^{\gamma}),   \quad \text{ and } \quad   x_n=O(n^{\frac{\mu}{2} }),      \quad \text{ as } n\to\infty\,.
\end{equation}
Then, for each $t\in\R^q\setminus\{0\}$,
\begin{equation}\label{eq:assnorm2}
\frac{t^\top (\hat\beta^{(l)}-\beta)}{\sigma\sqrt{t^\top  \left(
X_n^\top X_n\right)^{-1} t}} \tod Z\sim \mathcal{N}(0,1),\qquad \text{ as } n\to\infty.
\end{equation}
\end{theorem}
\begin{proof}
  By Corollary~\ref{cor:assnorm_cond2}, it suffices to show that
  \begin{equation*}
\lambda_{\min}(X_n^\top X_n)/ \| X_n \|_\infty^2\to \infty,
\quad \text{ as } n\to\infty.
\end{equation*}
  It is obvious from the monotonicity of the sequence $(x_i)_i$ that the maximal row sum of $X_n$ is $ \| X_n \|_\infty=1+x_n$ and since the assumptions imply $x_n\to\infty$, we have $ \| X_n \|_\infty\leq 2x_n$ for $n$ sufficiently large. Combining this with \eqref{eq:min-eigval-claim} and \eqref{eq:eigval-3}, and noting that $\bar x_n\leq x_n$, we obtain
  \begin{align*}
    \frac{\lambda_{\min}(X_n^\top X_n)}{ \| X_n \|_\infty^2}\geq \frac 12 \frac{\lambda_2^{(n)}}{(1+x_n)^2}\geq \frac 12 \frac{n}{(1+x_n)^2\left(1+\frac{1+\bar x_n^2}{\widetilde{S}_n^2}\right) }
    \geq \frac 12 \frac{n}{(1+x_n)^2+\frac{(1+ x_n^2)^2}{\widetilde{S}_n^2}}.
  \end{align*}
 The last expression tends to $\infty$ as $n\to\infty$, since, by assumption \eqref{eq:CondRadii2}, we have
  $$
  n^{-1}\left((1+x_n)^2+(1+ x_n^2)^2/\widetilde{S}_n^2)\right)\leq C_1 n^{\mu-1}+ C_2 n^{2\mu-\gamma-1}\,,
  $$
  for some constants $C_1,C_2$, which tends to $0$, since $\max\{\mu, 2\mu-\gamma\}<1$. This completes the proof.
\end{proof}

\begin{remark}\label{remRegrErrors}
To show strong consistency results, independence of discretization
and measurement errors $\delta_{kj}$ is required, cf.
\cite{drygas76}. Unfortunately, this assumption is not
rea\-lis\-tic in our case, since discretizations of $F_{\eps_j}$
clearly depend on each other for various $j$. Additionally,
intrinsic volumes $C_k(F_{\eps_j})$ are obviously dependent for
different $j$ and $k$.
\end{remark}

\subsection{Asymptotics of the second method}

Assume that the period $h_0>0$ and the detail level $m\in\N$ are known. For simplicity, we prove weak consistency results only
for each curvature measure $\sC_{k}$, $k=0,\ldots, d$ separately (separate regressions). So fix some $k\in\{0, \ldots, d\}$.  We fix the sequence of dilation radii to
be arithmetic at the logarithmic scale, i.e., $x_i=a_0+a\cdot i$, $i\in\N$ where $a_0\ge0$ and  $a>0$. Let $n\ge 2m+2$.
The regression \eqref{eq:regr_final} can be written in terms of the parameters
\begin{equation}\label{eq:beta_method2}
 \beta=(\beta_k,s,\alpha_1,\ldots,\alpha_m,\gamma_1,\ldots,\gamma_m )^\top
\end{equation}
with $\alpha_j=b_{kj} \cos \varphi_{kj}$, $\gamma_j=b_{kj} \sin \varphi_{kj}$, $j=1,\ldots, m$
as
\begin{equation*}\label{eq:regr_final_method2}
 y_{ki}=\beta_k+s \cdot x_i + \sum\limits_{j=1}^{m}\left(\alpha_j \cos (\mu_1 j x_i) - \gamma_j \sin (\mu_1 j x_i) \right)+ \delta_{ki}, \qquad i=1,\ldots, n\,,
\end{equation*}
with $\mu_1=2\pi / h_0$ and dependent errors $\{ \delta_{ki} \}$
of zero mean satisfying condition  \eqref{eq:Ql-ass}. This can be
equivalently rewritten in form \eqref{eq:Regression} with $l=n$,
$q=2m+2$, and the design matrix
\begin{equation*}
X_n=\begin{pmatrix} 1&x_1& \cos (\mu_1  x_1)& \cos (\mu_1 2 x_1)&   \ldots& \cos (\mu_1 m  x_1)& -\sin (\mu_1  x_1)& -\sin (\mu_1 2 x_1) & \ldots& -\sin (\mu_1 m x_1) \\
1&x_2& \cos (\mu_1  x_2)& \cos (\mu_1 2 x_2)&   \ldots& \cos (\mu_1 m  x_2)& -\sin (\mu_1  x_2)& -\sin (\mu_1 2 x_2)&  \ldots& -\sin (\mu_1 m x_2) \\
\vdots&\vdots & \vdots& \vdots &   \ldots& \vdots& \vdots& \vdots&  \ldots& \vdots \\
1&x_n& \cos (\mu_1  x_n)& \cos (\mu_1 2 x_n)&   \ldots& \cos
(\mu_1 m x_n)& -\sin (\mu_1  x_n)& -\sin (\mu_1 2 x_n)&  \ldots&
-\sin (\mu_1 m x_n)
\end{pmatrix}.
\end{equation*}

\begin{lemma}\label{lemm:Cons2MethRegression}
Let $F\subset\R^d$ be a set satisfying the assumptions (A1), (A2) and (A3'). Assume that $x_i=a_0+a\cdot i$, $i\in\N$, where $a_0\ge0$ and  $a>0$ such that $aj/h_0\notin\Z$ for $j=1,\ldots, 2m$. Then under the above conditions on the sequence of errors  $\{\delta_{kj} \}$,
the least squares estimator
$$   \hat{\beta}^{(n)}=( \hat{\beta}_{k}^{(n)}, \hat{s}_{k}^{(n)},\hat{\alpha}_{k,1}^{(n)},\ldots,\hat{\alpha}_{k,m}^{(n)},\hat{\gamma}_{k,1}^{(n)},\ldots,\hat{\gamma}_{k,m}^{(n)} )^\top $$
in \eqref{eq:LSEst} of the parameter vector \eqref{eq:beta_method2} is weakly consistent. 
\end{lemma}
\begin{proof}
Without loss of generality, we only consider the case  $a_0=0$,
$a=1$, that is, $x_i=i$, $i=1,\ldots, n$. (A constant $a_0\neq 0$
can be incorporated in the parameters $\alpha_{k,j}$ and
$\gamma_{k,j}$, and $a\neq 1$ can be included in the constant
$\mu_1$, such that the same arguments as below work for slightly
transformed parameters.) By Lemma~\ref{lemmDrygas}, it suffices to
show that $\lambda_{\min}(X_n^\top X_n)\to\infty$ as $n\to\infty$.
We claim that it is in fact sufficient to show that
\begin{align}
  \label{eq:trXnTXn-1}
  \tr\left((X_n^\top X_n)^{-1}\right)\to 0 \qquad \text{ as } n\to\infty\,.
\end{align}
Indeed, if $\lambda_1,\ldots,\lambda_{2m+2}$ are the eigenvalues of $X_n^\top X_n$ (which are all strictly positive since $X_n^\top X_n$ is positive definite), then $1/{\lambda_1},\ldots,1/\lambda_{2m+2}$ are the eigenvalues of $(X_n^\top X_n)^{-1}$ and we have
\begin{align*}
  \lambda_{\min}(X_n^\top X_n)=\min\limits_{j=1,\ldots ,2m+2} \lambda_j
  =\frac 1{\max_j(\frac{1}{\lambda_j})}
  \ge \frac 1{\sum_j \frac 1{\lambda_j}}
  = \frac 1{\tr\left((X_n^\top X_n)^{-1}\right)}\,,
\end{align*}
which tends to $+\infty$ as $n\to\infty$, if \eqref{eq:trXnTXn-1} holds.

Recall now that, by Cramer's rule, $\tr\left((X_n^\top X_n)^{-1}\right)$ is given by
\begin{align}
  \label{eq:trXnTXn-2}
  \tr\left((X_n^\top X_n)^{-1}\right)=\frac1{\det(X_n^\top X_n)}\sum_{j=1}^{2m+2} M_n^{j,j}\,,
\end{align}
where $M_n^{j,j}$ is the $(j,j)$ minor of $X_n^\top X_n$, $j=1,\ldots,2m+2$.
In the sequel we will show that, for each $j=1,\ldots,2m+2$,
\begin{align} \label{eq:minor_est}
  M_n^{j,j}=O(n^{2m+3}) \quad
  \text{ whereas } \quad
\det(X_n^\top X_n)=\Theta(n^{2m+4})  \quad \text{ as } \quad n\to\infty,
\end{align}
from which \eqref{eq:trXnTXn-1} follows at once.

The symmetric matrix $X_n^\top X_n=:(\xi_{jk})$ is given as follows
$$
X_n^\top X_n=\begin{pmatrix} A & V^\top & W^\top\\
 V & D &F^\top\\
 W & F & G
\end{pmatrix},\
$$
where
\begin{align*}
  A&=\begin{pmatrix}
    n & \sum_{i=1}^{n} x_i\\
    \sum_{i=1}^{n} x_i & \sum_{i=1}^{n} x_i^2
  \end{pmatrix}
  =\begin{pmatrix}
    n & n(n+1)/2\\
    n(n+1)/2 & n(n+1)(2n+1)/6
  \end{pmatrix}
  \in \R^{2\times 2},\\
  V&=(v_{j,k})\in \R^{m\times 2} \quad \text{ with } v_{j,1}=\sum_{i=1}^n \cos(\mu_1ij) \text{ and } v_{j,2}=\sum_{i=1}^n i\cdot\cos(\mu_1ij), \quad j=1,\ldots,m\,,\\
   W&=(w_{j,k})\in \R^{m\times 2} \quad \text{ with } w_{j,1}=-\sum_{i=1}^n \sin(\mu_1ij) \text{ and } w_{j,2}=-\sum_{i=1}^n i\cdot\sin(\mu_1ij),\quad j=1,\ldots,m\,,\\
   D&=(d_{j,k})\in \R^{m\times m} \quad \text{ with } d_{j,k}=\sum_{i=1}^n \cos(\mu_1ij)\cdot\cos(\mu_1ik), \quad j,k=1,\ldots,m \,,\\
   F&=(f_{j,k})\in \R^{m\times m} \quad \text{ with } f_{j,k}=-\sum_{i=1}^n \sin(\mu_1ij)\cdot\cos(\mu_1ik), \quad j,k=1,\ldots,m \,, \text{ and } \\
   G&=(g_{j,k})\in \R^{m\times m} \quad \text{ with } g_{j,k}=\sum_{i=1}^n \sin(\mu_1ij)\cdot\sin(\mu_1ik),\quad j,k=1,\ldots,m \,.
\end{align*}
Since we assumed $j/h_0\notin\Z$, the sums in the coefficients $v_{j,1}$ and $w_{j,1}$ can be simplified as follows, cf.\ e.g.~\cite[p.206]{Loebus}:
\begin{align*}
  v_{j,1}=\frac 12\left(\frac{\sin(\mu_1(n+\frac 12)j)}{\sin(\frac 12 \mu_1j)}-1\right) \quad \text{ and } \quad
  w_{j,1}=\frac 12 \frac{\cos(\mu_1(n+\frac 12)j)-\cos(\frac12{\mu_1j})}{\sin(\frac 12{\mu_1j})},\quad j=1,\ldots,m.
\end{align*}
The condition $j/h_0\notin\Z$ ensures also that $\sin(\frac 12
\mu_1j)=\sin(\pi \frac j{h_0})\neq 0$. Hence
\begin{align*}
  |v_{j,1}|\leq \frac 12\frac{|\sin(\mu_1(n+\frac 12)j)-\sin(\frac 12 \mu_1j)|}{|\sin(\frac 12 \mu_1j)|}\leq \frac{1}{|\sin(\frac 12{\mu_1j})|} \quad \text{ and } \quad
  |w_{j,1}|\leq \frac{1}{|\sin(\frac 12{\mu_1j})|}, \quad j=1,\ldots,m.
\end{align*}
This means that the coefficients $v_{j,1}$ and $w_{j,1}$ are
bounded from above and below by constants independent of $n$ for
each $j=1,\ldots,m$. In fact, they are all bounded by the same
constant
$\kappa:=\left(\min\limits_{j=1,\ldots,2m}|\sin(\frac12\mu_1j)|\right)^{-1}$.

Using the relations $\cos x \cos y=\frac 12
(\cos(x+y)+\cos(x-y))$, $\sin x \cos y=\frac 12
(\sin(x+y)+\sin(x-y))$ and $\sin x \sin y=\frac 12
\left(\cos(x-y)-\cos(x+y)\right)$ and  the above formulas, one
obtains analogously that the coefficients $d_{j,k}$, $f_{j,k}$ and
$g_{j,k}$ are bounded from above and below by constants
independent of $n$, whenever $j\neq k$ and for $f_{j,k}$ also in
the case $j=k$. This is ensured by the fact, that $j+k\leq 2m$ and
so, by the assumptions of the lemma, $(j-k)/h_0,
(j+k)/h_0\notin\Z$. In particular,
\begin{align*}
  f_{j,k}&=-\frac 12 \sum_{i=1}^n \sin(\mu_1 i(j+k))-\frac 12 \sum_{i=1}^n \sin(\mu_1i(j-k)),
\end{align*}
and so for $j=k$ the second sum on the right vanishes, while the first sum is absolutely bounded by $\kappa$ (similarly as $w_{j,1}$).
 Hence all entries of $X_n^\top X_n$ except those on the diagonal and in the second row and column are bounded absolutely by
 constant $\kappa$
 independent of $n$.
On the diagonal, we have similarly as for $v_{j,1}$ and $w_{j,1}$ 
\begin{align*}
d_{j,j}&=\frac 12 \sum_{i=1}^n \cos(0)+\frac 12 \sum_{i=1}^n
\cos(\mu_1i(2j))=\frac n2 + \frac 14\left(\frac{\sin(\mu_1(n+\frac
12)2j)}{\sin(\mu_1j)}-1\right), \quad j=1,\ldots, m,
\end{align*}
and
\begin{align*}
g_{j,j}&=\frac 12 \sum_{i=1}^n \cos(0)-\frac 12 \sum_{i=1}^n
\cos(c\mu_1 i(2j))=\frac n2 - \frac
14\left(\frac{\sin(\mu_1(n+\frac
12)2j)}{\sin(\mu_1j)}-1\right),\quad  j=1,\ldots, m.
\end{align*}
Hence, as $n\to\infty$, $\xi_{j,j}=\Theta(n)$ for $j\neq 2$ and $\xi_{2,2}=\Theta(n^3)$.
For the second column of $X_n^\top X_n$, we use 
\cite[Lemma 2.3(a), p.220]{Loebus} (which states that for any sequence $(a_n)_n$ of real numbers whose partial sums are bounded from above and below by a constant $C$, one has $|\sum_{i=1}^n i a_i|\leq 2n C$ ) to conclude that
\begin{align*}
  |v_{j,2}|=\left|\sum_{i=1}^n i\cdot\cos(\mu_1ij)\right|\leq 2\kappa n \quad \text{ and } \quad |w_{j,2}|=\left|\sum_{i=1}^n i\cdot\sin(\mu_1ij)\right|\leq 2\kappa n,
\end{align*}
for $ j=1,\ldots,m$.
Hence, as $n\to\infty$, $\xi_{j,2}=O(n)$ for $j=3,4,\ldots,2m$, while $\xi_{1,2}=\Theta(n^2)$ and $\xi_{2,2}=\Theta(n^3)$.

Having computed the order of growth of all coefficients of $X_n^\top X_n$, it is now easily seen, that $\det(X_n^\top X_n)=\Theta(n^{2m+4})$ as $n\to\infty$. Indeed, we have for the product of the diagonal entries $\prod_{j=1}^{2m+1} \xi_{j,j}=\Theta(n^{2m+4})$ as $n\to\infty$ and this product is the only term in the Leibnitz expansion of $\det(X_n^\top X_n)$ with this order of growth. All other terms are at most of the order of $n^{2m+3}$ as $n\to \infty$. Hence the order of growth cannot be reduced by cancellations with other terms.

For the $(j,j)$ minors $M_n^{j,j}$ of $X_n^\top X_n$ we can argue similarly. If the $j$-th row and column are deleted, in the remaining matrix the diagonal entries are still those with the maximal order of growth in each row and column. Hence the order of growth of the determinant $M_n^{j,j}$ is bounded by the product of the orders of its diagonal entries, that is $M_n^{j,j}=O(n^{2m+3})$ as $n\to\infty$ for each $j=1,\ldots, 2m$. (For $j=2$, we even have $M_n^{j,j}=O(n^{2m+1})$.) This completes the proof of
\eqref{eq:minor_est} and thus of the weak consistency of the estimator $\hat{\beta}^{(n)}$ as stated.
\end{proof}

\begin{theorem}\label{thm:consist}
Under the assumptions of Lemma \ref{lemm:Cons2MethRegression}, for any $k\in\{0,\ldots, d\}$,
the estimators $\hat{s}_{k}^{(n)}$ of $s$ and $|\widehat
\sC_{k}^{(n)}(F)|$ of $|\overline \sC_{k}(F)|$ are weakly
consistent as $n\to\infty$.
\end{theorem}
\begin{proof}
By Lemma \ref{lemm:Cons2MethRegression}, the estimators
$\hat{\beta}_{k}^{(n)}$, $ \hat{s}_{k}^{(n)}$ and $\hat{f}_{k}^{(n)}(x)$  are weakly
consistent as $n\to\infty$ in the
regression model \eqref{eq:regr_final}. More precisely, the estimators
$\hat{b}_{kj}^{(n)}$ and  $\hat{\varphi}_{kj}^{(n)}$
in \eqref{eq:EstFm} are given by
$$\hat{b}_{kj}^{(n)}=\sqrt{(\hat{\alpha}_{j}^{(n)})^2+(\hat{\gamma}_{j}^{(n)}) ^2}, \qquad
\hat{\varphi}_{kj}^{(n)}=\arctan{\left( \frac{\hat{\gamma}_{j}^{(n)}}{\hat{\alpha}_{j}^{(n)}}\right)},\quad j=1,\ldots, m.
$$
We split the estimation error into two parts as follows:
$$
|\widehat \sC_{k}^{(n)}(F)|-|\overline
\sC_{k}(F)|=I_{1,n}+I_{2,n},
$$
where by \eqref{eq:EstFracCurv}
$$
I_{1,n}= \frac {\exp\{\hat{\beta}_{k}^{(n)}\}}{h_0}
  \int_0^{h_0} \left( \exp\{ \hat{f}_{k}^{(n)}(x)\}   -   \exp\{ f_{k}(x)\} \right) \, {d} x,
$$
$$
I_{2,n}= \frac {\exp\{\hat{\beta}_{k}^{(n)}\}
-\exp\{\beta_{k}\}   }{h_0}
  \int_0^{h_0}  \exp\{ f_{k}(x)\} \, {d} x.
$$

To see the convergence  $I_{1,n} \stackrel{P}{\longrightarrow} 0$ as
$n\to\infty$, observe that the sequence
$\left(\hat{\beta}_{k}^{(n)}\right)_n$ converges to $\beta_{k}$ as $n\to\infty$ and is thus bounded.
Furthermore, $\hat{f}_{k}^{(n)}(x)
\stackrel{P}{\longrightarrow} f_{k}(x) $ for any
$x\in[0,h_0]$ as $n\to\infty$, and this convergence is uniform
with respect to $x$, since
\begin{multline}\nonumber
\left|\hat{f}_{k}^{(n)}(x)- f_{k}(x)\right|\le
\sum\limits_{j=1}^m\left( |\hat{b}_{kj}^{(n)}\cos
\hat{\varphi}_{kj}^{(n)} - b_{kj}\cos \varphi_{kj}
|\right.
+\left. |\hat{b}_{kj}^{(n)}\sin \hat{\varphi}_{kj}^{(n)}
 - b_{kj}\sin \varphi_{kj}| \right)=:\psi_n
\stackrel{P}{\longrightarrow} 0
\end{multline}
 as $n\to\infty$ for any $x\in[0,h_0]$. Noting that, for each $x\in[0,h_0]$ and each $n\in\N$, there exists a number $\xi_n(x)$ between $0$ and  $\hat f_k^{(n)}(x)-f_k(x)$ such that, by the mean value theorem,
 $$
 |e^{\hat f_k^{(n)}(x)}-e^{f_k(x)}|=|e^{f_k(x)}(e^{\hat f_k^{(n)}(x)-f_k(x)}-1)|=e^{f_k(x)} e^{\xi_n(x)}|\hat f_k^{(n)}(x)-f_k(x)|
 $$
 and that $\xi_n(x)\leq |\hat f_k^{(n)}(x)-f_k(x)|\leq \psi_n$ for each $x\in[0,h_0]$,
 we conclude that
 $$
|I_{1,n}|\le e^{\psi_n} \psi_n \int_0^{h_0}  \exp\{ f_{k}(x)\}\,
{d} x \stackrel{P}{\longrightarrow} 0  \quad \text { as } n\to\infty.
$$
For the convergence of $I_{2,n}$ simply observe that, by the continuous mapping theorem,
$$
 |e^{\hat \beta_k^{(n)}}-e^{\beta_k}|
 \stackrel{P}{\longrightarrow} 0 \quad \text{ as } n\to \infty\,.
 $$
\end{proof}

\begin{remark}
We believe that under the assumptions of Theorem
\ref{thm:consist}, also the estimators $\tilde{s}^{(n)}$,
$|\widetilde \sC_{k}^{(n)}(F)|$, $k=0,\ldots, d$ in the case of
simultaneous regression are weakly consistent as $n\to\infty$ and
that this can be proved with essentially the same arguments as in
the case of separate regression in the proof of Theorem
\ref{thm:consist}. However, in view of the rather long and
technical arguments in the `easy' case of separate regression, we
did not attempt to verify all the details.
\end{remark}

\begin{remark}
 The fractal curvature estimates $\exp(\hat \beta^{(n)}_k)$ only rarely deviate significantly from the estimates obtained by the second method.
Theoretically, since the $\hat \beta^{(n)}_k$ are now obtained through averaging over $\log |p_k(e^{-x})|$ plus some errors, an application of Jensen's inequality to the convex function $\exp$ suggests that the estimates of the absolute value of fractal curvature are now systematically too small; practically, however, this discrepancy between first and second method is not visible in the examples we consider in Section \ref{sec:numerical}.
\end{remark}



 Let $m\in \N$ be fixed and $h_0>0$ unknown.
In \cite[Theorem 1'']{Han73}, strong consistency of $\hat
\beta^{(n)}$ (after estimation and subtraction of the linear trend
described in Section \ref{subsec:estimators}, i.e. setting
formally $\beta_k=s=0$) as well as of $\hat{h}_0$ estimated by
\eqref{estim_h0} is proven under the assumption that $\{
\delta_{ki} \}$ forms a stationary regular sequence with zero
mean. Ergodicity of $\{ \delta_{ki} \}$ together with further
assumptions such as e.g. the continuity of its spectral density
$f_\delta$ imply the asymptotic normality of $\widehat{\mu}_j$ and
$\hat \beta^{(n)}$, see \cite[Theorem 2]{Han73}. A law of iterated
logarithm for $\widehat{\mu}_j$ is given in
\cite[p.~57]{QuiHan01}.

Now let $m$ be unknown. If $\{ \delta_{ki}\}$ is a stationary
Gaussian linear process with known positive spectral density
$f_{\delta}$ then an a.s. consistent estimate of $m$ (as
$n\to\infty$) is given in \cite[Theorem 15, p.~75]{QuiHan01}. Its
idea is to set $\widehat{m}$ to be the smallest possible value of
$m$ such that the log likelihood of $\{ \tilde{y}_{ki}\}$
decreases when gradually reducing $m$. For $\{ \delta_{ki}\}$
being an AR process with Gaussian innovations, see
\cite[p.~80]{QuiHan01}.

If parameters $h_0$ and $m$ are consistently estimated then the
consistency of the estimators of the fractal curvatures can be proven
similarly as in Theorem \ref{thm:consist}.
\section{Numerical simulations and results}
\label{sec:numerical}

\paragraph{Binary images of fractals.}
We assume that binary images consist of pixels which belong to the
rectangular grid $\mathbb Z^2$, endowed with the Euclidean metric
inherited from $\mathbb R^2$. This means that the distance between
neighbouring pixels is $1$, which we henceforth adopt as the unit
of length. Pixels can assume the two values $0$ (white) and $1$
(black). A binary image is a map from the lattice $\mathbb Z^2$ to
the set $\{0,1\}$. We say that a binary image $\tilde F$ is a
discretization of a subset $F\subset\mathbb R^2$ if, for any pixel
$(k,l) \in \mathbb Z^2$, $\tilde F(k,l)=1$, whenever the square
$[k,k+1) \times [l,l+1)$ has non-empty intersection with $F$.

Binary images of self-similar sets can easily be generated on a
computer using iterated function systems; for algorithms see e.g.\ \cite{barnsley2000fe}.
For the generation of the sample images in this paper we have used the
free software \emph{Fractal Explorer} \cite{fractalexplorer}.
We have generated binary images of three arithmetic and three non--arithmetic fractals on a 3000 by 3000 pixel canvas (see Figure \ref{fig:thesets}).

\begin{figure}[ht]
\centering
\subfigure[Sierpi\'nski Gasket]{
\includegraphics[width=0.3\textwidth]{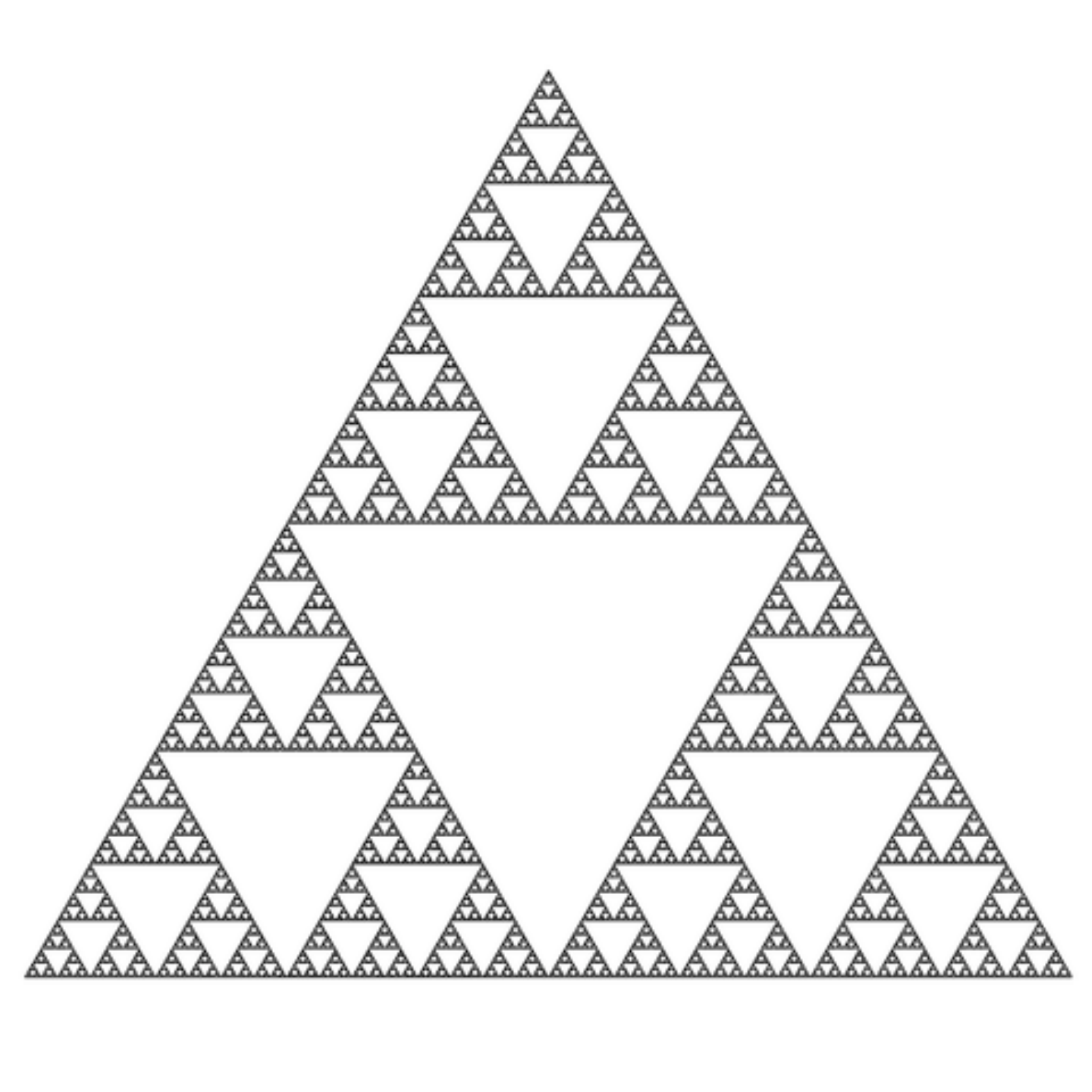}
}
\subfigure[Sierpi\'nski Carpet]{
\includegraphics[width=0.3\textwidth]{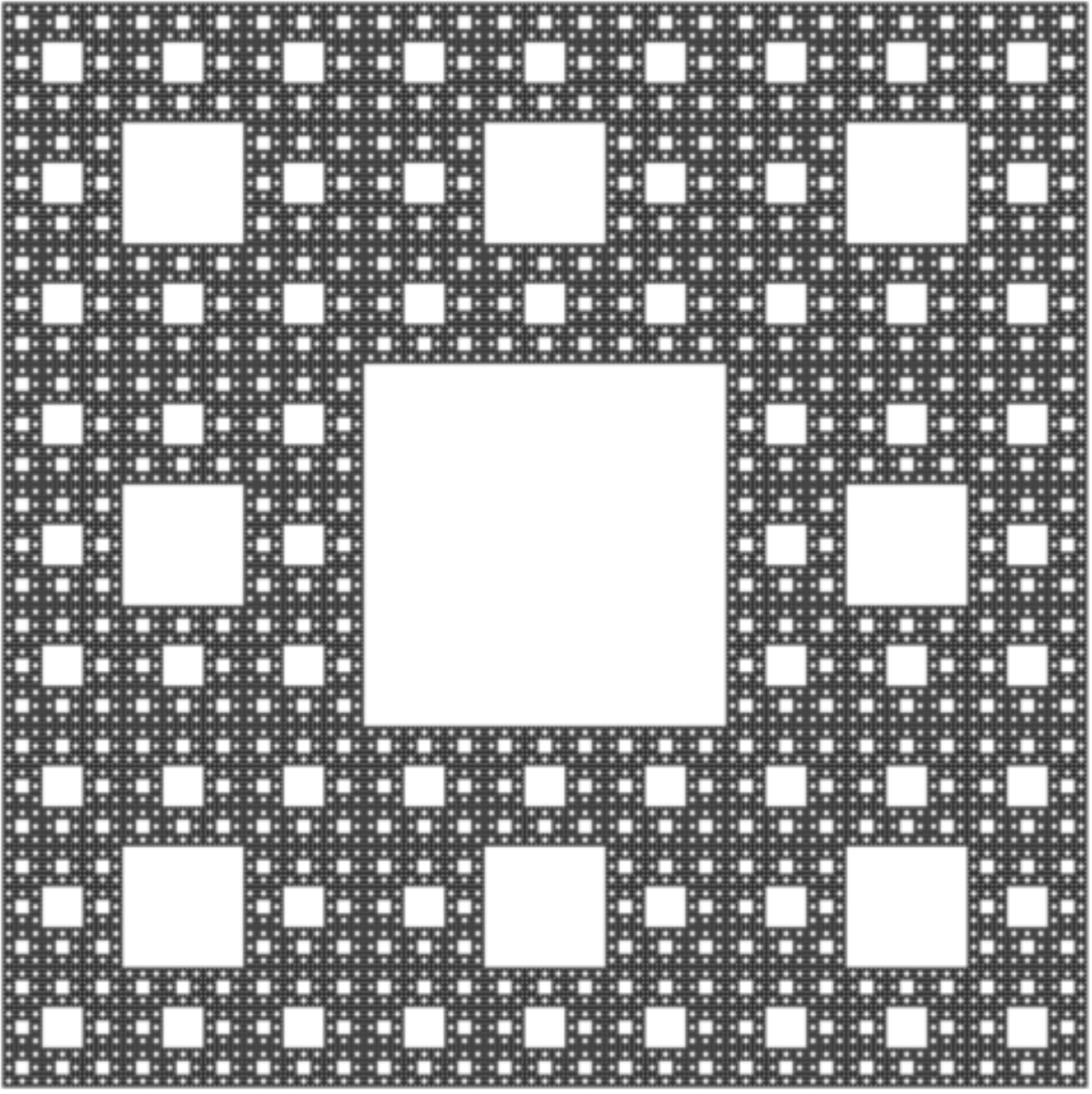}
}
\subfigure[modified Sierpi\'nski Carpet]{
\includegraphics[width=0.3\textwidth]{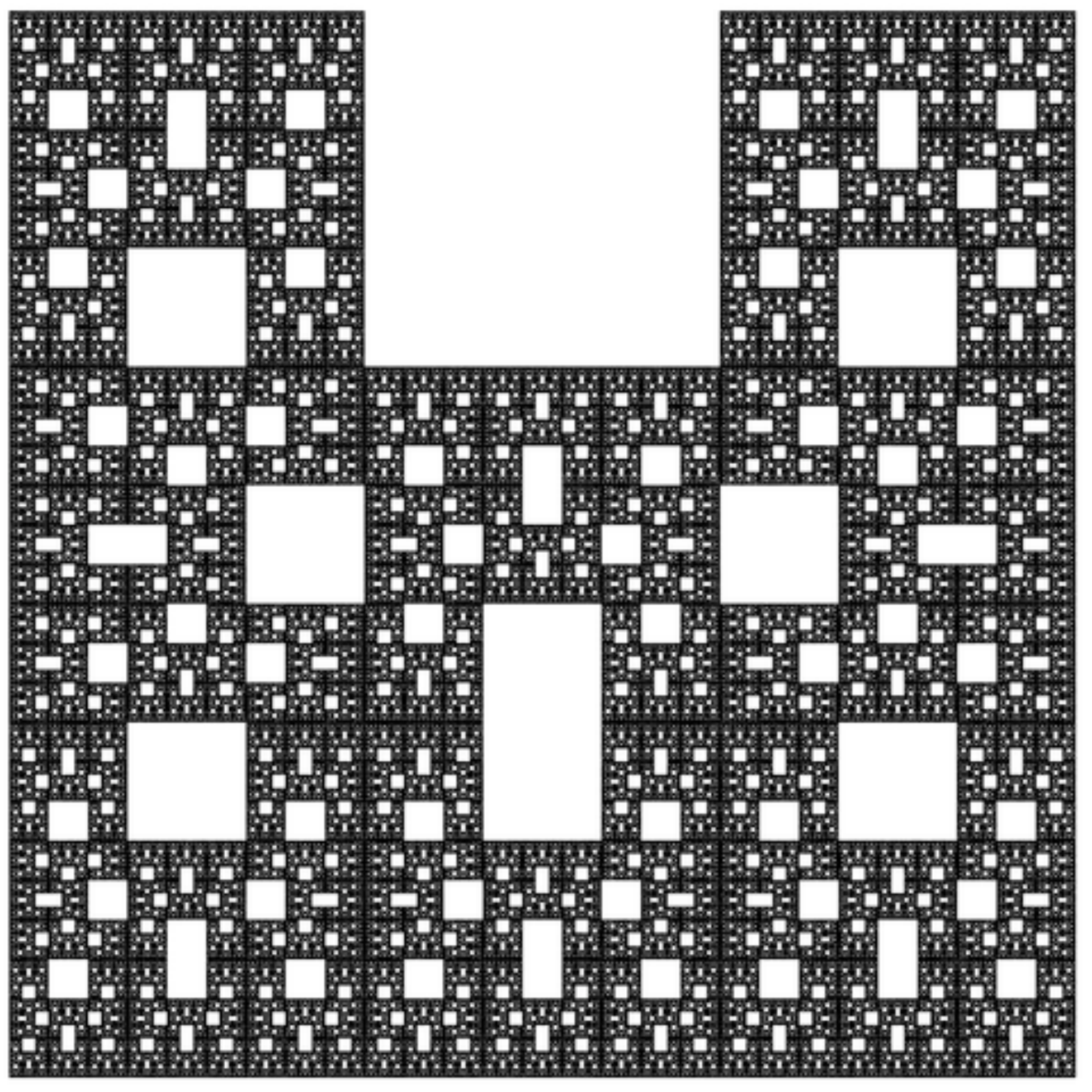}
}
\subfigure[Triangle]{
\includegraphics[width=0.3\textwidth]{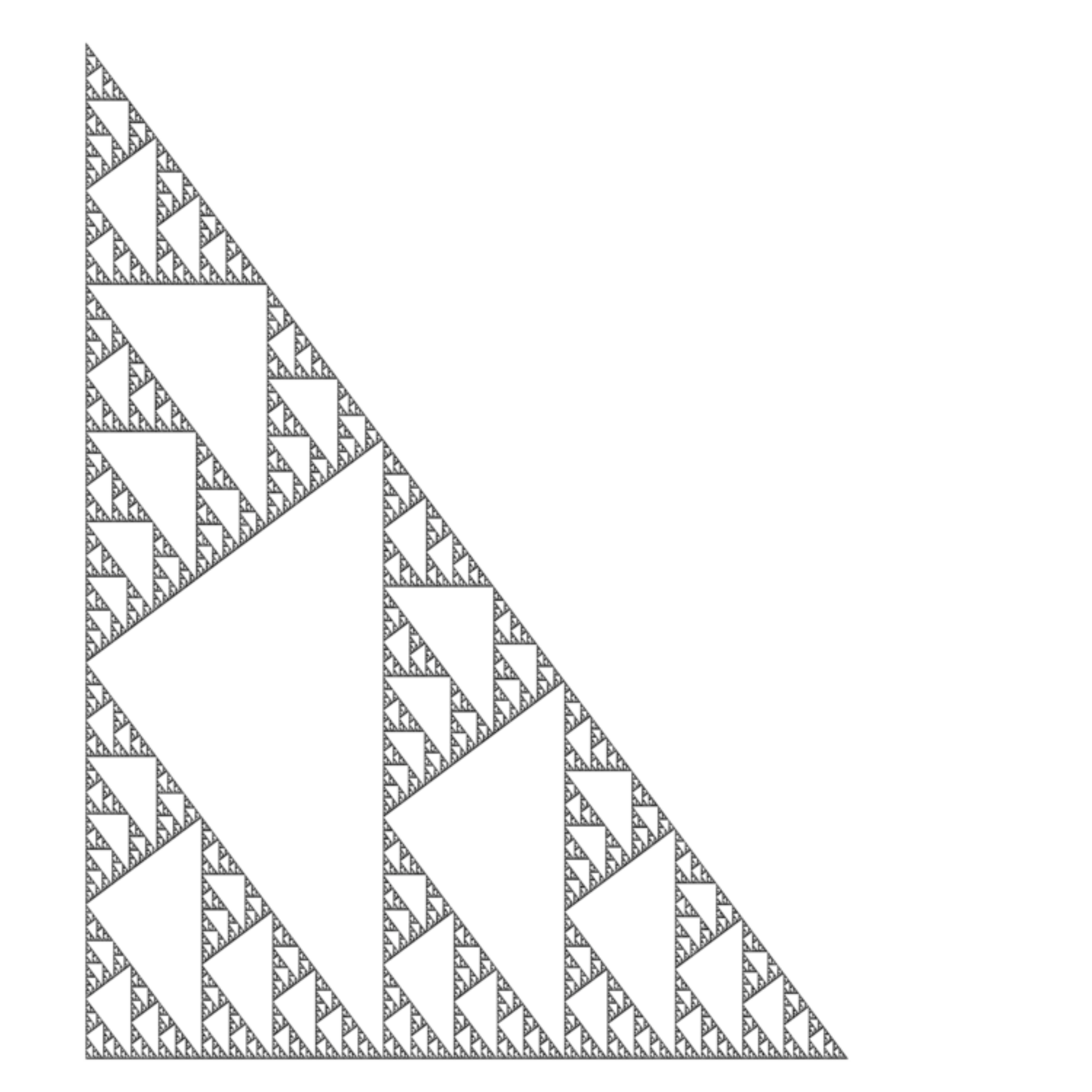}
}
\subfigure[Cross set]{
\includegraphics[width=0.3\textwidth]{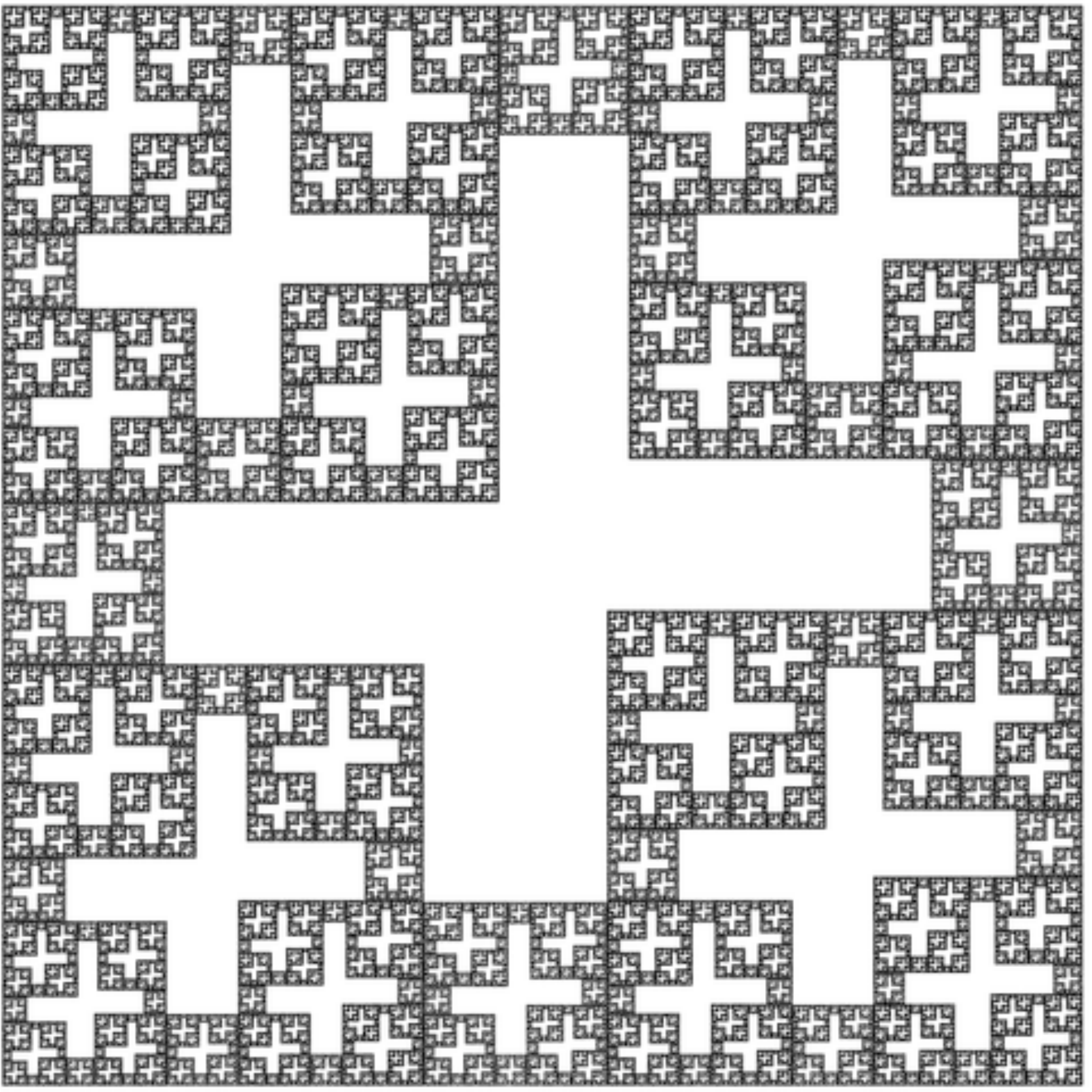}
}
\subfigure[Supergasket]{
\includegraphics[width=0.3\textwidth]{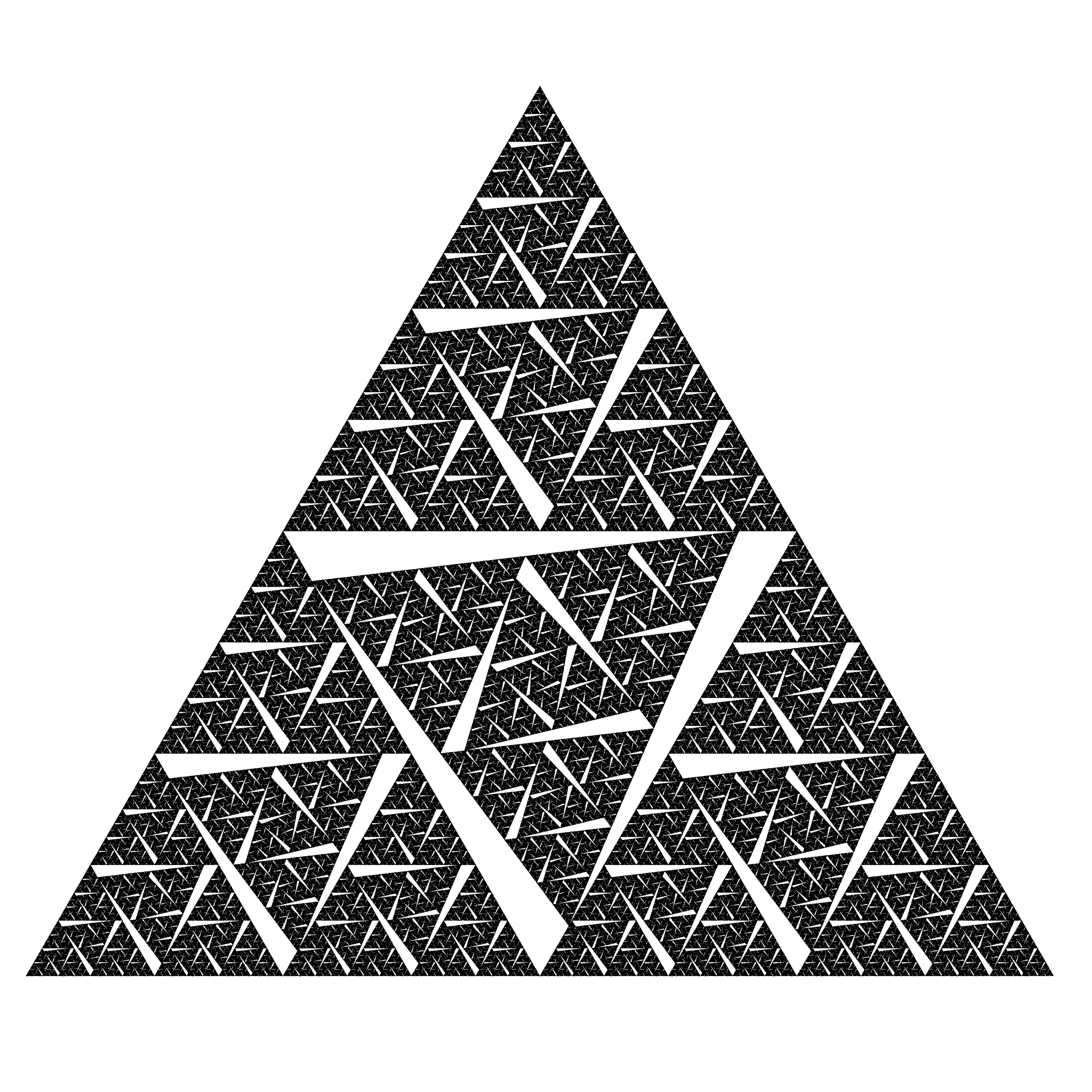}
} \caption{The sample fractals. Sets (a) -- (c) are arithmetic and
sets (d) -- (f) are non-arithmetic.} \label{fig:thesets}
\end{figure}

\paragraph{Obtaining the data.}
Let $\tilde F$ be a discretized fractal set. For $\eps >0$,
we approximate the dilated set $F_\eps$ by the dilated binary images
$\tilde F_\eps$, which we calculate as follows (cf.~e.g.~\cite{soille03}): First, compute the
distance transform of $\tilde F$,
  \begin{eqnarray*}
  D_{\tilde F} : \mathbf \mathbb Z^2 &\rightarrow & \R \\
                p &\mapsto  & d\left( p,F^{-1}\left( \{1\}\right) \right),
  \end{eqnarray*}
which records the distance of each pixel on the canvas to the nearest black pixel.
Then $\tilde F_\eps$ results from setting all pixels $p$ to black which satisfy $D_F(p)\leq \eps$.

We generated discretized dilated images $\tilde F_{\eps_i}$ for a
set of dilation radii $\eps_i = e^{-x_i}$. The $x_i$ were evenly
spaced with distance $0.02$ ranging from around $-4.5$ to $-1$
(which correspond to radii $\eps_i$ ranging from $87$ down to
$2.7$). This seemed feasible, as for too large $\eps >
e^{4.5}\approx 90$ the scaling behaviour of the intrinsic volumes
approached that of a full $2$-dimensional set, and for too small
$\eps < e \approx 2.71$ the discretization errors were too large.
We note that especially the choice of the largest dilation radius
$\eps_1$ needs to be adapted to each fractal $F$, since there is
no good a priori choice: If $\eps_1$ is too small this will result
in a shortage of data, whereas a too large $\eps_1$ will distort
the estimates.

We note that there is a set $\left \{\sqrt{{1}/{\pi}}, \sqrt{{5}/{\pi}}, \sqrt{{9}/{\pi}}, \sqrt{{37}/{\pi}},\ldots \right\}$ of radii which is special
in the sense that discrete and continuous disks with these radii have the same area.
Stoyan \cite{stoyan1994frs} recommends this choice of radii for the sausage method, and it might also be considered for the methods discussed in
this paper, especially if only a small set of data is to be collected due to computational limitations.

The next step is to measure the intrinsic volumes
$C_k(F_{\eps_j})$ for each $\eps_j$ and each $k=0,\ldots,d$. We
employ the algorithms described in \cite{minkdiscr} and
\cite{guderleietal} which determine for a fixed set $F_{\eps_j}$ all functionals $C_k(F_{\eps_j})$, $k=0, \ldots, d$ simultaneously. The relevant data set of $y_{ki}$-values is then
determined according to equation (\ref{eq:def-variables}).

\paragraph{The estimates.}
We have implemented the simultaneous linear least squares
regression estimators (LRE) from eqs.\ \eqref{eq:sJ} and
\eqref{eq:beta_kJ} and the simultaneous non-linear least squares
regression estimators (NRE) of the second method as given by
\eqref{eq:regr_final}, which were then included in the software
libraries of project \emph{GeoStoch} \cite{geostoch} of Ulm
University. We applied both LRE and NRE to the data set of each
fractal, regardless of whether it was an arithmetic or a
non-arithmetic set. The resulting estimates for the fractal
dimension and fractal curvatures are collected in
Tables~\ref{tab:dim} and \ref{tab:curv}.

\begin{table}[ht]
\centering
\begin{tabular}{r@{ }l|*{6}l}
 & &\includegraphics[width=1cm]{gasket.pdf}
&\includegraphics[width=1cm]{carpet.pdf}
&\includegraphics[width=1cm]{modcarpet.pdf}
&\includegraphics[width=1cm]{triangle.pdf}
&\includegraphics[width=1cm]{quadrate.pdf}
& 
\includegraphics[width=1cm]{barlow.pdf}
\\ \hline
& exact & 1.585 & 1.893 & 1.893 & 1.588 & 1.794 & 1.893 \\ \hline
 &box-counting& 1.54 & 1.88 & 1.83 & 1.57 & 1.78 & 1.84 \\
 LRE,& $J=\{0,1,2\}$ & 1.584 & 1.87 & 1.84 & 1.576 & 1.83 & 1.88 \\
NRE,& $J=\{0,1,2\}$& 1.587 & 1.87 & 1.85 & 1.576 & 1.83 & 1.88 \\
LRE,& $J=\{2\}$ & 1.586 & 1.87 & 1.86 & 1.589 & 1.78 & 1.889 \\
NRE,& $J=\{2\}$& 1.586 & 1.87 & 1.86 & 1.589 & 1.78 & 1.889 \\
 LRE,& $J=\{1\}$ & 1.558 & 1.8 & 1.7 & 1.57 & 1.73 & 1.85 \\
NRE,&  $J=\{1\}$ & 1.558 & 1.8 & 1.71 & 1.57 & 1.74 & 1.85 \\
 LRE,&  $J=\{0\}$ & 1.61 & 1.95 & 1.95 & 1.57 & 1.98 & 1.89 \\
 NRE,&  $J=\{0\}$ & 1.6 & 1.93 & 1.99 & 1.57 & 1.94 & 1.88\\[1mm]
\end{tabular}
\caption{Estimates of fractal dimension. The first row contains
the known exact dimension of each fractal, rounded to three
decimals. The set $J$ describes the orders $k$ of intrinsic
volumes $C_k$ used in the estimate. LRE and NRE refer to the first
and second method from Section~\ref{sec:estimator}, respectively.
For the method NRE, the period $h_0$ was estimated from the data,
and the detail parameter was chosen as $m=4$.}
\label{tab:dim}
\end{table}

The results suggest that for dimension estimation, LRE and NRE perform equally well. The dimension estimates based on boundary length data only ($k=1$)
 and Euler characteristic data only ($k=0$) are less reliable than estimates based on the volume data ($k=2$), which corresponds to the sausage method.
The dimension estimates based on all three intrinsic volumes ($k\in\{0,1,2\}$), however, seem to be comparable in accuracy to the method ``$k=2$''
and the standard box counting method, for which we used the free software FracLac \cite{fraclac}.

\begin{table}[ht]
 \centering
\begin{tabular}{rr|*5{|r}}
&& & $s$ known, & $s, h_0$ known, & $s$ unknown, & $s, h_0$ unknown\\
&&exact&LRE&NRE&LRE&NRE\\ \hline \hline
\multirow{3}{*}{\includegraphics[width=1cm]{gasket.pdf}}
&k=0&-13197&-10868&-11389&-10848&-11386\\
&k=1&117230&124471&124557&124235&123251\\
&k=2&564100&572880&572845&571798&566781\\ \hline
\multirow{3}{*}{\includegraphics[width=1cm]{carpet.pdf}}
&k=0&-58716&-47745&-58126&-45242&-54107\\
&k=1&262770&363432&364825&344376&339293\\
&k=2&4900200&5210885&5209660&4937666&4847744\\ \hline
\multirow{3}{*}{\includegraphics[width=1cm]{modcarpet.pdf}}
&k=0&-50742&-41177&-47178&-35060&-36587\\
&k=1&260960&363062&361158&309123&312815\\
&k=2&4871275&5095192&5092666&4338213&4418800\\ \hline
\multirow{3}{*}{\includegraphics[width=1cm]{triangle.pdf}}
&k=0&-9843&-8828&N/A&-8555&-8554\\
&k=1&100416&104144&N/A&100919&100908\\
&k=2&487649&495583&N/A&480237&480184\\ \hline
\multirow{3}{*}{\includegraphics[width=1cm]{quadrate.pdf}}
&k=0&?&-17555&N/A&-19366&-19377\\
&k=1&?&381454&N/A&420805&420571\\
&k=2&?&3387112&N/A&3736520&3735336\\ \hline
\multirow{3}{*}{\includegraphics[width=1cm]{barlow.pdf}}
&k=0&-9388&-16261&N/A&-15677&-15681\\
&k=1&159663&147590&N/A&142288&142156\\
&k=2&2497116&2513942&N/A&2423634&2421933\\[1mm]
\end{tabular}
\caption{Estimates of the $k$-th fractal curvatures for $k=0,1,2$.
The first column contains the exact value of the corresponding fractal curvature, rounded to accuracy $1$, which was available for all sets except the Cross set.
Columns two and three contain the LRE and NRE estimates based on the knowledge of the true regression parameters $s$ (dimension) and $h_0$ (period). (Since non-arithmetic sets do not have a period, no values appear for those sets in column three.)
For NRE, the detail level parameter $m$ was chosen as 4; this seemed reasonable as increasing $m$ further changed the estimates only very slightly.
 Columns four and five contain simultaneous LRE and NRE estimates of all curvatures, where the dimension $s$ and the period $h_0$ were also estimated
  as explained in Section \ref{subsec:estimators}.}
\label{tab:curv}
\end{table}

Estimates of the fractal curvatures typically result in a relative error of around $10\%$ to $20\%$. An exception is the $0$-th curvature of the supergasket, which is rather dramatically overestimated. The problem are the pointed angles in this set, which lead to large discretization errors for the Euler characteristic.

We remind the reader that in both methods (NRE and LRE), fractal curvatures and fractal dimension are estimated \emph{simultaneously}. In order to
 test the stability of curvature estimates with respect to the dimension estimate, we have compared the outputs of NRE and LRE to their outputs
 conditional on a known fractal dimension $s$ (s.\ Table~\ref{tab:curv}).  Noticable differences were only found for the modified Sierpi\'nski carpet.
We interpret this as some evidence for the curvature estimates being reasonably stable with respect to errors in the dimensional estimate.

Moreover, we noticed that for the arithmetic fractals the periodicity was by far more evident in the Euler characteristic than in the boundary length
 or area, which explains why the differences between the two methods are most apparent for the $0$-th curvature estimate. This is consistent with the
 observation that the peaks in the periodograms of the time series of Euler characteristics are more pronounced than the peaks of the other time series, see Figure~\ref{fig:perio}, making the Euler characteristic a useful data set for the estimate of the period of arithmetic fractals.

Finally, we remark that non-arithmetic fractals yield virtually the same output for both NRE and LRE models. Hence NRE should be preferred over LRE whenever there is some doubt about whether a self-similar fractal is arithmetic or not.

In the examples, we have included three different sets of equal
dimension, namely the two carpets
(b) and (c) and the supergasket (f), cf.\ Figure~\ref{fig:thesets}. The structure of the sets (b) and (c) is rather similar, while the set (f) looks very different. The differences in the geometry are also visible in the fractal curvatures. While the fractal curvatures of (b) and (c) only differ slightly, those of the set (c) take completely different values. One can easily distinguish (f) from the other two using any of the estimated fractal curvatures. The sets (b) and (c) are best distinguished by the estimated
fractal Euler number, compare Table~\ref{tab:curv}.

\section{Summary and outlook}

We have introduced two methods for estimating the fractal dimension and the fractal curvatures of a given fractal set based on binary images. We have shown the consistency of our methods under suitable assumptions on the covariance structure of the errors and the choice of the radii. We have implemented and tested these methods on some toy examples of self-similar sets. While for the estimation of the fractal dimension our methods perform equally well as the standard methods, such as box counting, we obtain at the same time estimates of the fractal curvatures which we propose to use as additional geometric characteristics for image classification. The theoretical background provided by singular curvature theory is a strong argument for using these characteristics in favour of or in addition to other texture parameters suggested in the literature.

We point out that our consistency results only show that the suggested estimators estimate indeed the  fractal curvatures if the
resolution goes to infinity and the sequence of radii tends to zero in a suitable way. We make no claim about how well our estimators
perform if the resolution is kept finite, that is, in any scenario found in practice. Also, we did not address at all the question of how
well the implemented algorithms perform with respect to computational costs or running time. We have implemented our methods in the most
 obvious way, computing the intrinsic volumes for each dilation radius separately, for which each time a scan of the whole image is
 necessary. This allowed to use for this step existing algorithms in the \emph{GeoStoch} library \cite{geostoch}. Probably, a lot of optimization is possible in the step of determining the intrinsic volumes of the parallel sets. It may be possible to obtain the curvature data of all parallel sets in a single scan of the image.

Notice that so far the theoretical foundations (that is, the
existence of fractal curvatures) are laid for fractal sets
exhibiting some form of self-similarity, including
 strictly self-similar sets \cite{diss-winter,Zaehle,RZ12}, self-conformal sets \cite{Kombrink,Bohl13} and also some random self-similar fractals, as described in \cite{Zaehle}. For fractal sets exhibiting some weaker form of self-similarity, similar results are expected to hold and therefore the described methods may be used whenever some form of self-similarity is present.
However, one should be aware that for general (random) fractals
$F$ of dimension $s$, the (expected or almost sure) scaling
exponents $s_k(F)$ might not  necessarily be equal to $s-k$ or if
they are, that the fractal curvatures may not exists, not even the
averaged versions. For the Brownian path in $\R^d$,  $d\ge 2$, for
instance, the fractal dimension is $s=2$ (almost surely and in the
mean) and the scaling exponents are $s_k=s-k$ for the volume
($k=d$) and the surface area ($k=d-1$) for all dimensions $d>2$,
cf.\ \cite{RaSchmSpo09, ratwin1, HR12}. For $d=2$, however,
the corresponding fractal curvatures are zero because the correct
scaling is $-\log \varepsilon$ for the area $C_2(F_\varepsilon)$
(almost surely and in the mean) and
 $\varepsilon \log^2 \varepsilon$ for the perimeter $2 C_1(F_\varepsilon)$ as $\varepsilon \to
 0$ (at least in the mean).

Against this background, it is important to note that the
suggested algorithms may also be used as a test to check whether
the hypothesis $s_k=s-k$ (implied by (A3) and (A3')) is satisfied
for a given set and some $k$. For this purpose simply a separate
regression for the index $k$ (that is with $J=\{k\}$ in the sense
of Remark~\ref{rem:J}) can be carried out and the estimate of the
fractal dimension can be compared to the dimension estimate of the
simultaneous regression or to one of the sausage method
($J=\{d\}$). It is for instance not too difficult to check that
the parallel sets of the Koch curve have Euler characteristic $1$,
which means $C_0(F_\eps)=1$ for each $\eps>0$. Hence a separate
regression for $k=0$ applied to an image of the Koch curve $F$
should find an estimate for $s_k(F)$ very close to $0$. This is
indeed what we observed. Also the violation of assumption (A2) can
easily be checked from the data and the relevant indices can be
excluded from the estimation.

\paragraph{\bf Acknowledgements} The authors would like to thank
Martina Z\"ahle for stimulating discussions on fractals and
geometric measure theory. During the work on this article SW was supported by a DFG grant, project no.\  WI 3264/2-2.

\bibliographystyle{amsplain}  
\bibliography{references}

\end{document}